\documentclass[12pt, english]{amsart}
\usepackage[all]{xy}
\usepackage{varioref,babel}
\usepackage{enumerate,a4,amsfonts,amsmath,amsfonts,amssymb,amscd}
\usepackage{amsxtra,eucal}
\usepackage[latin1]{inputenc}
\usepackage{tabularx}

\usepackage{mathrsfs}
\usepackage{hyperref}
\usepackage{mathtools}

\usepackage[foot]{amsaddr}

\newcommand{\remove}[1]{} 

\newcommand{\N}{\mathbb{N}}
\newcommand{\Q}{\mathbb{Q}}

\newcommand{\Z}{\mathbb{Z}}

\newcommand{\calP}{{\mathcal{P}}}

\newcommand{\calR}{{\mathcal{R}}}

\newcommand{\catC}{{\mathscr{C}}}

\newcommand{\catG}{{\mathscr{G}}}

\newcommand{\SMatII}[4]{\left[\begin{array}{cc} {#1} & {#2} \\ {#3} &
{#4} \end{array}\right]}
\newcommand{\smallSMatII}[4]{\left[\begin{smallmatrix} {#1} & {#2} \\ {#3} &
{#4} \end{smallmatrix}\right]}

\newcommand{\invlim}{\underleftarrow{\lim}\,}

\DeclareMathOperator{\Char}{char}

\DeclareMathOperator{\End}{End} %

\DeclareMathOperator{\Hom}{Hom} %
\DeclareMathOperator{\id}{id}

\DeclareMathOperator{\Jac}{Jac} %
\newcommand{\op}{\mathrm{op}}

\newtheorem{thm}{Theorem}[section]
\newtheorem{lem}[thm]{Lemma}
\newtheorem{prp}[thm]{Proposition}
\newtheorem{cor}[thm]{Corollary}
\newtheorem{dfn}[thm]{Definition}

\newtheorem{baseexample}[thm]{Example} 
\newtheorem{baseremark}[thm]{Remark} 

\newenvironment{example}
{\begin{baseexample}\rm}{\end{baseexample}}
\newenvironment{remark}
{\begin{baseremark}\rm}{\end{baseremark}}

\newcommand{\units}[1]{{#1^\times}}   

\newcommand{\WGr}[2]{{\mathrm{WG}}^{#1}(#2)}      
\newcommand{\Witt}[2][]{\mathrm{W}^{#1}(#2)}      
\newcommand{\WittF}[2][]{\mathrm{W}^{#1}(#2)}     
\newcommand{\Sesq}[1]{\mathrm{Sesq}(#1)}          
\newcommand{\SesqF}[1]{#1}                        
\newcommand{\SysSesq}[2]{\mathrm{Sesq}_{#2}(#1)}  
\newcommand{\Herm}[2][]{\mathrm{UH}^{#1}(#2)}     
\newcommand{\Hyp}[2][]{\mathbb{H}^{#1}_{#2}}      
\newcommand{\Transfer}[1]{\mathrm{T}_{#1}}        
\newcommand{\scalarExt}[2]{\calR_{{#1}/{#2}}}     

\newcommand{\DA}[2][]{{\mathrm{Ar}}_{2#1}({#2})}                    
\newcommand{\TDA}[2][]{{\mathrm{A\widetilde{r}}}_{2#1}({#2})}       

\newcommand{\WGrS}[1]{\mathrm{WG}_S(#1)}          
\newcommand{\WittS}[1]{\mathrm{W}_S(#1)}          

\newcommand{\rMod}[1]{{\mathrm{Mod}\textrm{-}{#1}}} 

\newcommand{\rproj}[1]{\calP(#1)}                   

\newcommand{\rAd}[1]{{#1}_r}      
\newcommand{\lAd}[1]{{#1}_\ell}   

\title[Hermitian Categories]{Hermitian categories, extension of scalars and systems of sesquilinear forms}

\author{Eva Bayer-Fluckiger$^1$}
\author{Uriya A.\ First$^2$}
\author{Daniel A.\ Moldovan$^1$}
\address{$^1$\'{E}cole Polytechnique F\'{e}d\'{e}rale de Lausanne, Switzerland.}
\address{$^2$Hebrew University of Jerusalem, Israel.}

\date{\today}

\thanks{
The second named author is partially supported by an Israel-US BSF grant \#2010/149 and an ERC grant \#226135.
The third named author was partially supported by the Swiss National Science Fundation, grant 200020-109174/1.
}



\begin{document}

\maketitle

\begin{abstract}
    We prove that the category of \emph{systems of sesquilinear forms} over a given hermitian
    category is equivalent to the category of \emph{unimodular $1$-her\-mi\-tian forms} over another hermitian category.
    The sesquilinear forms
    are not required to be unimodular or defined on a reflexive object (i.e.\
    the standard map from the object to its double dual is not assumed to be bijective), and the forms in the system
    can be defined with respect to different hermitian structures on the given category. This extends a result
    obtained in \cite{BayerMoldovan12}.

    We use the equivalence to define a Witt group of sesquilinear forms over a
    hermitian category, and also to  generalize various results (e.g.: Witt's Cancelation Theorem,
    Springer's Theorem, the weak Hasse principle, finiteness of genus) to systems
    of sesquilinear forms over hermitian categories.
\end{abstract}

\section*{Introduction}

Quadratic and hermitian forms were studied extensively by various authors, who have developed a rich
array of
tools to study them.
It is well-known that in many cases (e.g.\ over fields), the theory of sesquilinear forms can be reduced to the theory
of hermitian forms (e.g.\ see \cite{Ri74}, \cite{RiSh76} and works based on them).
In the recent paper
\cite{BayerMoldovan12}, an explanation of this reduction was provided in the form of an equivalence
between the category of sesquilinear  forms over a ring and the category of unimodular $1$-hermitian forms over a
special hermitian category.

In this paper, we extend the equivalence of \cite{BayerMoldovan12} to hermitian categories, and
moreover, improve it in such a way that it applies to systems of sesquilinear forms in hermitian
categories that admit \emph{non-reflexive} objects (see section~\ref{section:hermitian-categories}).
That is, we prove that the category of systems of sesquilinear forms over a hermitian category $\catC$ is equivalent
to the category of unimodular $1$-hermitian forms over anther hermitian category $\catC'$. The sesquilinear forms
are not required to be unimodular or defined on a reflexive object, and  the forms in the system
can be defined with respect to different hermitian structures on the category $\catC$.

Using the equivalence, we present a notion of a Witt group of sesquilinear forms, which
is analogous to the standard Witt group of hermitian forms over rings with involution
(e.g.\ see \cite{Kn91} or \cite{Sch85QuadraticAndHermitianForms}).
We also extend various results (Witt's Cancelation Theorem, Springer's Theorem, finiteness
of genus, the Hasse principle, etc.) to systems of sesquilinear forms over
hermitian categories (and in particular to systems of sesquilinear forms over rings with a family of involutions).

\medskip

Sections~\ref{section:sesquilinear-forms} and~\ref{section:hermitian-categories} recall
the basics of sesquilinear forms over rings and hermitian categories, respectively.
In section~\ref{section:equivalence-of-categories}, we prove the equivalence of the category of
sesquilinear forms over a given hermitian category
to a category of unimodular $1$-hermitian forms over another hermitian category, and in section~\ref{section:systems},
we extend this result to systems of sesquilinear forms. Section~\ref{section:applications} presents
applications of the equivalence.

\section{Sesquilinear and Hermitian Forms}
\label{section:sesquilinear-forms}

Let $A$ be a ring. An \emph{involution} on $A$ is an additive map  $\sigma : A \to A$
such that $\sigma(ab) = \sigma(b) \sigma(a)$ for all $a,b \in A$ and $\sigma^2=\id_A$.
Let $V$ be a right $A$-module.
A \emph{sesquilinear form} over $(A,\sigma)$ is a biadditive map
$s : V \times V \to A$ satisfying $s(xa,yb) = \sigma(a) s(x,y) b$ for all $x,y \in V$ and  $a, b \in A$.
The pair $(V,s)$ is also called a sesquilinear form in this case.\footnote{
    Some texts use the term \emph{sesquilinear space}.
}
The \emph{orthogonal sum} of two sesquilinear forms $(V,s)$ and $(V',s')$
is defined to be $(V \oplus V', s \oplus s')$ where $s\oplus s'$ is given by
$$
(s \oplus s')(x \oplus x', y \oplus y') =
s(x,y) + s'(x',y')
$$
for all $x, y \in V$ and $x', y' \in V'$. Two sesquilinear forms $(V,s)$ and $(V', s')$ are called \emph{isometric} if
there exists an isomorphism of $A$-modules $f: V \xrightarrow{\sim} V'$ such that $s'(f(x), f(y))=s(x,y)$ for all $x, y \in V$.

Let $V^* = \Hom_A(V,A)$. Then $V^*$ has a right $A$-module structure given by
$(f\cdot a)(x) = \sigma(a) f(x)$ for all $f \in V^*$, $a \in A$. We say that $V$ is
\emph{reflexive} if the homomorphism of right $A$-modules $\omega_V : V \to V^{**}$ defined by $\omega_V(x)(f) = \sigma(f(x))$ for all $x \in V$, $f \in V^*$
is bijective.

A sesquilinear space $(V,s)$ over $(A, \sigma)$ induces two homomorphisms of right $A$-modules $\lAd{s},\rAd{s}:V \to V^*$ called the
\emph{left} and \emph{right adjoint} of $s$, respectively.
They are given by  $\lAd{s}(x)(y) = s(x,y)$ and
$\rAd{s}(x)(y) = \sigma(s(y,x))$ for all $x, y \in V$. Observe that $\rAd{s}=\lAd{s}^* \omega_V$ and
$\lAd{s}=\rAd{s}^*\omega_V$. The form $s$ is called \emph{unimodular} if $\rAd{s}$ and $\lAd{s}$ are isomorphisms.
In this case, $V$ must be reflexive.

Let $\epsilon=\pm 1$. A sesquilinear form $(V,s)$ over $(A, \sigma)$ is called \emph{$\epsilon$-hermitian} if
$\sigma(s(x,y))=\epsilon s(y,x)$ for all $x, y \in V$, i.e. $\rAd{s}=\epsilon \lAd{s}$. A $1$-hermitian form is also called a hermitian form.

\medskip

There exists a classical notion of Witt group for unimodular $\epsilon$-hermitian forms over $(A, \sigma)$ (e.g.\ see  \cite{Kn91}): Denote by
$\WGr{\epsilon}{A, \sigma}$ the Grothendieck group
of isometry classes of unimodular $\epsilon$-hermitian forms $(V,s)$ over $(A, \sigma)$ with $V$ finitely generated
projective, the addition being  orthogonal sum.
A unimodular $\epsilon$-hermitian form
over $(A, \sigma)$ is called \emph{hyperbolic} if it is isometric to $(V\oplus V^*, \Hyp[\epsilon]{V})$ for
some finitely generated projective right $A$-module $V$, where $\Hyp[\epsilon]{V}$ is defined by
\[
\Hyp[\epsilon]{V}(x\oplus f,y\oplus g)=f(y)+\epsilon \sigma(g(x))\qquad\forall x,y\in V,\,\, f,g\in V^*\ .
\]
We let $\Hyp{V}=\Hyp[1]{V}$.
The quotient of $\WGr{\epsilon}{A, \sigma}$ by the subgroup generated by the unimodular $\epsilon$-hermitian hyperbolic forms is called the
\emph{Witt group of unimodular $\epsilon$-hermitian forms}
over $(A, \sigma)$ and is denoted by $\Witt[\epsilon]{A, \sigma}$.

\medskip

We denote by $\Sesq{A,\sigma}$ (resp.\ $\Herm[\epsilon]{A,\sigma}$) the category of sesquilinear (resp.\ unimodular $\epsilon$-hermitian) forms over $(A, \sigma)$.
The morphisms of these categories are (bijective) isometries. For simplicity, we let $\Herm{A,\sigma}:=\Herm[1]{A,\sigma}$.

\section{Hermitian Categories}
\label{section:hermitian-categories}

This section recalls some basic notions about hermitian categories as presented in \cite{Sch85QuadraticAndHermitianForms} (see also \cite{Kn91},
\cite{QuSchSch79}).

\subsection{Preliminaries}

Recall that a \emph{hermitian category} consists of a triple $(\catC,*,\omega)$ where
$\catC$ is an additive category, $*:\catC\to \catC$ is a contravariant functor and
$\omega=(\omega_C)_{C\in\catC}:\id\to **$ is a natural transformation  satisfying
$\omega^*_C  \omega_{C^*}= \id_{C^*}$ for all $C \in {\catC}$. In this case, the pair $(*,\omega)$
is called a \emph{hermitian structure} on $\catC$. It is customary to assume
that $\omega$ is a natural \emph{isomorphism} rather than a natural \emph{transformation}. Such hermitian categories
will be called \emph{reflexive}. In general, an object $C\in\catC$ for which $\omega_C$ is an isomorphism
is called \emph{reflexive}, so the category $\catC$ is reflexive precisely when all its objects are reflexive.
We will often drop $*$ and $\omega$ from the notation and use these symbols
to denote the functor and natural transformation associated
with any hermitian category under discussion.

A \emph{sesquilinear form} over the category ${\catC}$ is a pair $(C, s)$ with $C\in {\catC}$
and $s: C \rightarrow C^*$. A sesquilinear form $(C, s)$ is called \emph{unimodular} if $s$
and $s^* \omega_C$ are isomorphisms. (If $C$ is reflexive, then $s$ is bijective
if and only if $s^* \omega_C$ is bijective.)
Let $\epsilon=\pm 1$. A sesquilinear form $(C,s)$ is called \emph{$\epsilon$-hermitian}
if $s= \epsilon s^* \omega_C$. For brevity, $1$-hermitian forms are often called hermitian forms. Orthogonal sums of
forms are defined in the obvious way. Let $(C,s)$ and $(C',s')$ be two sesquilinear forms over ${\catC}$.
An \emph{isometry} from $(C,s)$ to $(C',s')$ is an isomorphism
$f : C \xrightarrow{\sim} C'$ satisfying $s = f^* s' f$. In this case, $(C,s)$ and $(C',s')$
are said to be \emph{isometric}. We let $\Sesq{\catC}$ stand for the category of sesquilinear
forms over $\catC$ with isometries as morphisms.

Denote by $\Herm[\epsilon]{\catC}$ the category of unimodular $\epsilon$-hermitian forms over $\catC$.
The morphisms are isometries. For brevity, let $\Herm{\catC}:=\Herm[1]{\catC}$.
The hyperbolic unimodular $\epsilon$-hermitian forms over $\catC$ are the forms isometric to $(Q\oplus Q^*,\Hyp[\epsilon]{Q})$, where
$Q$ is any reflexive object in $\catC$
and $\Hyp[\epsilon]{Q}$ is given by
\[
\Hyp[\epsilon]{Q}=
\SMatII{0}{id_{Q^*}}{\epsilon\omega_Q}{0}:Q\oplus Q^*\to (Q\oplus Q^*)^*=Q^*\oplus Q^{**}\ .
\]
Again, let $\Hyp{Q}=\Hyp[1]{Q}$.
The quotient of $\WGr{\epsilon}{\catC}$, the
Grothendieck group of isometry classes of unimodular $\epsilon$-hermitian forms over $\catC$ (w.r.t.\ orthogonal sum),
by the subgroup generated by the hyperbolic forms is called the Witt group of
unimodular $\epsilon$-hermitian forms over $\catC$ and is denoted by $\Witt[\epsilon]{\catC}$. For brevity, set $\Witt{\catC}=\Witt[1]{\catC}$.

\begin{example}\label{EX:hermitian-category-obtained-from-a-ring-with-involution}
    Let $(A,\sigma)$ be a ring with involution. If we take $\catC$
    to be $\rMod{A}$, the category of  right $A$-modules,
    and define $*$ and $\omega$ as in section \ref{section:sesquilinear-forms},
    then $\catC$ becomes a hermitian category.
    Furthermore, the sesquilinear forms $(M,s)$ over $(A,\sigma)$ correspond to the sesquilinear forms
    over $\catC$ via $(M,s)\mapsto (M,\rAd{s})$. This correspondence gives rise to isomorphisms of categories
    $\Sesq{A,\sigma}\cong\Sesq{\catC}$ and
    $\Herm[\epsilon]{A,\sigma}\cong \Herm[\epsilon]{\catC}$.
    Now let $\catC$ be a subcategory of $\rMod{A}$ such that $M\in\catC$ implies
    $M^*\in \catC$. Then $\catC$ is still a hermitian category and
    it is reflexive if and only if $\catC$ consists of reflexive $A$-modules (as defined in section \ref{section:sesquilinear-forms}).
    For example, this happens if $\catC=\rproj{A}$, the category of
    projective $A$-modules of finite type. In this case, the Witt group
    $\Witt[\epsilon]{\catC}=\Witt[\epsilon]{\rproj{A}}$ is isomorphic to $\Witt[\epsilon]{A,\sigma}$.
\end{example}

\subsection{Duality Preserving Functors}

Let $\catC$ and $\catC'$ be two hermitian categories. A \emph{duality preserving functor} from $\catC$ to
$\catC'$ is an additive functor
$F: \catC \rightarrow \catC'$ together with a natural isomorphism $i=(i_M)_{M \in \catC}:F* \to *F$. This means that for any $M \in \catC$, there exists
an
isomorphism $i_M:F(M^*) \xrightarrow{\sim} (FM)^*$ such that for all $N \in \catC$ and $f \in \Hom_{\catC}(M,N)$, the
following diagram commutes:
$$\xymatrix@!C{
F(N^*) \ar[r]^{F(f^*)} \ar[d]^{i_N} & F(M^*) \ar[d]^{i_M}\\
(FN)^* \ar[r]^{(Ff)^*} & (FM)^*
}$$
Any duality preserving functor induces a functor $\Sesq{\catC}\to\Sesq{\catC'}$, which we also denote
by $\SesqF{F}$. It is given by
$$\SesqF{F}(M, s) = (FM, i_M F(s))$$
for every $(M,s)\in\Sesq{\catC}$. If the functor $F:\catC\to\catC'$ is faithful (resp.\ faithful and full, induces an equivalence),
then so is the functor $\SesqF{F}:\Sesq{\catC}\to \Sesq{\catC'}$.

Let $\lambda=\pm 1$. A duality preserving functor $F$ is called
\emph{$\lambda$-hermitian} if
\[i_{M^*} F(\omega_M)=\lambda i^*_M \omega_{FM}\]
for all $M \in \catC$. Let $\epsilon=\pm 1$.
We recall from
\cite[pp.\ 80-81]{Kn91}  that in this case,
the functor $\SesqF{F}:\Sesq{\catC}\to\Sesq{\catC'}$ maps $\Herm[\epsilon]{\catC}$
to $\Herm[\epsilon \lambda]{\catC'}$ and sends $\epsilon$-hermitian hyperbolic forms
to $\epsilon\lambda$-hermitian hyperbolic forms. Therefore, $F$ induces a homomorphism between the corresponding Witt groups:
$$\WittF[\epsilon]{F}: \Witt[\epsilon]{\catC} \rightarrow \Witt[\epsilon \lambda]{\catC'}.$$
If $F$ is an equivalence of categories,
then $\Sesq{F}:\Herm[\epsilon]{\catC} \to \Herm[\epsilon \lambda]{\catC'}$
is also an equivalence of categories and the induced group homomorphism $\WittF[\epsilon]{F}$ is an isomorphism of groups.

\subsection{Transfer into the Endomorphism Ring}
\label{subsection:transfer}

The aim of this subsection is to introduce the method of \emph{transfer into the endomorphism ring}, which allows us to pass from the abstract setting of hermitian categories to that of a ring with involution,
which is more concrete. This method will be  applied repeatedly in section \ref{section:applications}.
Note that it applies well only to reflexive hermitian categories.

\medskip

Let $\catC$ be a \emph{reflexive} hermitian category, and let $M$ be an
object of $\catC$, on which we suppose that there exists a unimodular $\epsilon_0$-hermitian form $h_0$ for a certain $\epsilon_0=\pm 1$.
Put $E=\End_{\catC}(M)$.
According to \cite[Lm.\ 1.2]{QuSchSch79}, the form $(M,h_0)$ induces on $E$ an involution $\sigma$, defined by
$\sigma(f)=h_0^{-1} f^* h_0$ for all $f \in E$. Let $\rproj{E}$ denote the category of projective right $E$-modules
of finite type. Then, using $\sigma$, we can consider $\rproj{E}$
as a reflexive hermitian category
(see Example~\ref{EX:hermitian-category-obtained-from-a-ring-with-involution}).

Recall that an \emph{idempotent $e \in \End_{\catC}(M)$ splits} if there exist an object $M' \in \catC$ and morphisms
$i:M' \rightarrow M$, $j: M \rightarrow M'$ such that $ji=\id_{M'}$ and $ij=e$.

Denote by $\catC\vert_M$ the full subcategory of $\catC$ consisting of objects of $\catC$ which are isomorphic to a direct summand
of a finite direct sum of copies of $M$. We consider the following functor:
$$\Transfer{}=\Transfer{(M,h_0)}:=\Hom(M,\,\underline{~{}~}\,): \catC\vert_M \rightarrow \rproj{E}$$
$$N \mapsto \Hom(M, N), \qquad\forall N \in \catC\vert_M$$
$$f \mapsto \Transfer{}(f), \qquad\forall f \in {\rm Hom}(N,N'), ~\forall N,N' \in \catC\vert_M,$$
where for all $g \in {\rm Hom}(M,N)$, $\Transfer{}(f)(g)=fg$.
In \cite[Pr.\ 2.4]{QuSchSch79}, it has been proven that the functor $\Transfer{}$
is fully faithful and duality preserving with respect to the natural isomorphism
$i=(i_N)_{N \in \catC\vert_M}:\Transfer{}* \to *\Transfer{}$ given by $i_N(f)=\Transfer{}(h_0^{-1}f^* \omega_N)$ for every $N \in \catC\vert_M$ and $f \in {\rm Hom}(M,N^*)$. In addition,
if all the idempotents of $\catC\vert_M$ split, then $\Transfer{}$ is an equivalence of categories. By computation, we easily see that $\Transfer{}$ is $\epsilon_0$-hermitian.

Note that for any finite list of (reflexive) objects $M_1,\dots, M_t\in\catC$ and any $\epsilon_0=\pm 1$,
there exists a unimodular $\epsilon_0$-hermitian form $(M,h_0)$ such that $M_1,\dots,M_t\in \catC\vert_{M}$.
Indeed, let $N= \bigoplus_{i=1}^tM_i$ and take $(M,h_0)=(N\oplus N^*,\Hyp[\epsilon_0]{N})$.
This means that as long as we treat finitely many hermitian forms, we may pass to the context of hermitian
forms over rings with involution.

\subsection{Linear Hermitian Categories and Ring Extension}
\label{subsection:ring-extension}

In this subsection we introduce the notion of extension of rings in hermitian categories.

\medskip

Let $K$ be a commutative ring. Recall that a \emph{$K$-category} is an additive category $\catC$ such
that for every $A,B\in\catC$, $\Hom_{\catC}(A,B)$ is endowed with a $K$-module structure such that the composition  is $K$-bilinear.
For example, any additive category is in fact a $\Z$-category.
An additive covariant functor $F:\catC\to\catC'$ between two $K$-categories is \emph{$K$-linear} if the map
$F:\Hom_{\catC}(A,B)\to \Hom_{\catC'}(FA,FB)$ is $K$-linear for all $A,B\in\catC$. $K$-linear contravariant functors
are defined in the same manner.
A \emph{$K$-linear hermitian category} is a hermitian category $(\catC,*,\omega)$ such that $\catC$ is a $K$-category
and $*$ is $K$-linear.

Fix a commutative ring $K$.
Let $\catC$ be an additive $K$-category and let $R$ be a $K$-algebra (with unity, not necessarily commutative).
We define the \emph{extension of the category $\catC$ to the ring $R$}, denoted $\catC \otimes_K R$, to be
the category whose objects are formal symbols $C\otimes_K R$ with $C\in \catC$ and its $\Hom$-sets are defined by
\[
\Hom_{\catC\otimes_KR}(A\otimes_K R,B\otimes_K R)=\Hom_{\catC}(A,B)\otimes_C R\ .
\]
The composition in $\catC\otimes_KR$ is defined in the obvious way.
It is straightforward to check that $\catC\otimes_K R$ is also a $K$-category. Moreover, when $R$ is commutative,
$\catC\otimes_K R$ is an $R$-category.
We define the \emph{scalar extension functor}, $\scalarExt{R}{K}:\catC\to \catC\otimes_K R$ by
$$\scalarExt{R}{K}M = M\otimes_K R, \qquad\forall M \in \catC\qquad\textrm{and}$$
$$\scalarExt{R}{K}f = f \otimes_K 1, \qquad\forall f \in \Hom(M,N).$$
The functor $\calR_{R/K}$ is additive and $K$-linear.

In case $K$ is obvious from the context, we write $\catC_R$, $M_R$, $f_R$ instead of
$\catC\otimes_K R$, $M\otimes_K R$, $f\otimes_K 1$, respectively. (Here, $M\in\catC$
and $f$ is a morphism in $\catC$.)

\begin{remark}\label{RM:scalar-extension-in-module-categories}
    The scalar extension we have just defined agrees with scalar extension of modules under mild assumptions, but
    not in general: Let $S$ and $R$ be two $K$-algebras and write $S_R=S\otimes_K R$.
    There is an additive functor
    $G:(\rMod{S})_R\to \rMod{(S_R)}$  given by
    \[G(M_R)=M\otimes_{S}S_R\qquad\textrm{and}\]
    \[G(f\otimes a)(m\otimes b)=fm\otimes ab\]
    for all $M,N\in\rMod{S}$, $f\in\Hom_S(M,N)$, $a,b\in R$,
    and the following diagram commutes
    \[
    \xymatrixcolsep{5pc}\xymatrix{
    \rMod{S} \ar[r]^-{\scalarExt{R}{K}} \ar@{=}[d]& (\rMod{S})_R \ar[d]^G \\
    \rMod{S} \ar[r]^-{\underline{~{}~}\,\otimes_S S_R} & \rMod{(S_R)}
    }
    \]
    In general, $G$ is neither full nor faithful. However, using standard tensor-$\Hom$
    relations, it is easy to verify that
    the map
    \begin{equation}\label{EQ:scalar-extension-in-module-categories}
    G:\Hom_{(\rMod{S})_R}(M_R,M'_R)\to \Hom_{\rMod{(S_R)}}(GM_R,GM'_R)
    \end{equation}
    is bijective if
    either (a) $M$ is finitely generated projective or (b)
    $R$ is a flat $K$-module and $M$ is finitely presented.
    In particular, if $\catC$ is an additive subcategory of $\rMod{S}$ consisting of finitely presented modules
    and $R$ is flat as a $K$-module, then $\catC_R$ can be understood as a full subcategory
    of $\rMod{(S_R)}$ in the obvious way.
    An example in which
    the map $G$ of \eqref{EQ:scalar-extension-in-module-categories} is neither injective nor
    surjective can be obtained by taking
    $S=K=\Z$, $R=\Q$ and $M=M'=\Z[\frac{1}{p}]/\Z$.
\end{remark}

If $(\catC,*,\omega)$ is a $K$-linear hermitian category and $R/K$ is a \emph{commutative} ring extension,
then $\catC_R$ also has a hermitian structure
given by $(M_R)^*=(M^*)_R$,
$(f \otimes a)^*=f ^* \otimes a$ and $\omega_{M_R}=(\omega_M)_R=\omega_M\otimes 1$
for all $M, N \in \catC$, $f \in \Hom_{\catC}(M, N)$ and
$a \in R$. In this case, the functor $\scalarExt{R}{K}$ is a $1$-hermitian duality preserving functor (the
natural transformation $i:\scalarExt{R}{K}*\to *\scalarExt{R}{K}$ is just the identity).
In particular, we get a functor
$\SesqF{\scalarExt{R}{K}}:\Sesq{\catC}\to\Sesq{\catC_R}$
given by $\SesqF{\scalarExt{R}{K}}(M,s):=(M_R,s_R)$ and
$\SesqF{\scalarExt{R}{K}}$ sends $\epsilon$-hermitian (hyperbolic) forms
to $\epsilon$-hermitian (hyperbolic) forms.

\subsection{Scalar Extension Commutes with  Transfer}
\label{subsection:scalar-extension-commutes-with-transfer}

Let $R/K$ be a commutative ring extension, let $\catC$ be a reflexive $K$-linear hermitian
category and let $M$ be an object of $\catC$ admitting a unimodular $\epsilon$-hermitian form $h$.
Then $(M_R,h_R)$ is a unimodular $\epsilon$-hermitian  form over $\catC_R$.
Let $E=\End_{\catC}(M)$ and $E_R=\End_{\catC_R}(M_R)=E\otimes_K R$.
It is easy to verify that the following diagram (of functors) commutes
\[
\xymatrixcolsep{5pc}\xymatrix{
\catC\vert_M \ar[r]^{\Transfer{(M,h)}} \ar[d]^{\scalarExt{R}{K}} & \rproj{E} \ar[d]^{\underline{~{}~}\,\otimes_E E_R} \\
\catC_R\vert_{M_R} \ar[r]^{\Transfer{(M_R,h_R)}}  & \rproj{E_R}
}\ .
\]
(Note that by Remark~\ref{RM:scalar-extension-in-module-categories}, $\rproj{E_R}$  and $\underline{~{}~}\,\otimes_E E_R$
can be understood as $\rproj{E}_R$ and $\scalarExt{R}{K}$, respectively.)
Since all the functors are $\epsilon$- or $1$-hermitian, we get the following commutative
diagram, in which the horizontal arrows are full and faithful,
\[
\xymatrixcolsep{5pc}\xymatrix{
\Herm[\lambda]{\catC\vert_M} \ar[r]^{\SesqF{\Transfer{(M,h)}}} \ar[d]^{\SesqF{\scalarExt{R}{K}}} &
\Herm[\lambda\epsilon]{\rproj{E}} \ar[d]^{\underline{~{}~}\,\otimes_E E_R} \\
\Herm[\lambda]{\catC_R\vert_{M_R}} \ar[r]^{\SesqF{\Transfer{(M_R,h_R)}}}  & \Herm[\lambda\epsilon]{\rproj{E_R}}
}\ .
\]
This diagram means that in order to study the behavior of $\scalarExt{R}{K}$ on arbitrary
$K$-linear hermitian categories, it is enough to study its behavior on hermitian categories obtained from $K$-algebras
with $K$-involution (as in Example \ref{EX:hermitian-category-obtained-from-a-ring-with-involution}).

\section{An Equivalence of Categories}
\label{section:equivalence-of-categories}

Let $\catC$ be a (not-necessarily reflexive) hermitian category.
In this section we prove that there exists
a \emph{reflexive} hermitian category
$\catC'$ such that the category $\Sesq{\catC}$ is equivalent to $\Herm[1]{\catC'}$.
(We explain how to extend this result to \emph{systems of sesquilinear forms}  in the next section.)

The category $\catC'$ resembles the category of
double arrows presented in \cite[\S3]{BayerMoldovan12}, but is not
identical to it. This difference makes our construction work for non-reflexive hermitian categories
and, as we shall explain in the next section, for systems of sesquilinear forms, where
the forms can be
defined with respect to  different hermitian structures on $\catC$.

\subsection{The Category of Twisted Double Arrows}

Let $(\catC,*,\omega)$ be a hermitian category. We construct
the \emph{category of twisted double arrows in $\catC$}, denoted $\TDA{\catC}$, as follows:
The objects of $\TDA{\catC}$ are quadruples $(M,N,f,g)$ such that $f,g\in\Hom_{\catC}(M,N^*)$.
A morphism from $(M,N,f,g)$ to $(M',N',f',g')$ is a  pair
$(\phi,\psi^\op)$ such that $\phi\in\Hom(M,M')$, $\psi\in\Hom(N',N)$, $f'\phi=\psi^*f$ and $g'\phi=\psi^*g$.
The composition of two morphisms is given by $(\phi,\psi^\op)(\phi',\psi'^\op)=(\phi\phi',(\psi'\psi)^\op)$.

The category $\TDA{\catC}$ is easily seen to be an additive category. Moreover, it
has a hermitian structure: For every $(M,N,f,g)\in\TDA{\catC}$, define $(M,N,f,g)^*=(N,M,g^*\omega_N,f^*\omega_N)$
and $\omega_{(M,N,f,g)}=\id_{(M,N,f,g)}=(\id_M,\id_N^\op)$. In addition, for every
morphism $(\phi,\psi^\op):(M,N,f,g)\to(M',N',f',g')$, let
$(\phi,\psi^\op)^*=(\psi,\phi^\op)$.
It is now routine to check
that $(\TDA{\catC},*,\omega)$ is a \emph{reflexive} hermitian category.
Also observe that $**$ is  just the identity functor on $\TDA{\catC}$.
The following proposition describes the hermitian forms
over $\TDA{\catC}$.

\begin{prp}\label{PR:sesquilinear-forms-in-the-category-of-double-arrows}
    Let $Z:=(M,N,f,g)\in\TDA{\catC}$ and let $\alpha,\beta\in \Hom_{\catC}(M,N)$. Then
    $(Z,(\alpha,\beta^\op))$ is a  hermitian form over $\TDA{\catC}$ $\iff$ $\alpha=\beta$
    and $\alpha^* f=g^*\omega_N\alpha$ $\iff$ $\alpha=\beta$
    and $\alpha^* g=f^*\omega_N\alpha$.
\end{prp}

\begin{proof}
    By definition, $Z^*=(N,M,g^*\omega_N,f^*\omega_N)$, so $(\alpha,\beta^\op)$ is morphism
    from $Z$ to $Z^*$ if and only if $\beta^* f=g^*\omega_N\alpha$ and
    $\beta^* g=f^*\omega_N\alpha$.
    In addition, by computation, we see that $(\alpha,\beta^\op)=(\alpha,\beta^\op)^*\circ\omega_Z$
    precisely when $\alpha=\beta$. Therefore, $(Z,(\alpha,\beta^\op))$ is a hermitian
    form if and only if $\alpha=\beta$, $\alpha^* f=g^*\omega_N\alpha$ and
    $\alpha^* g=f^*\omega_N\alpha$. It is therefore enough to show
    $\alpha^* f=g^*\omega_N\alpha$ $\iff$
    $\alpha^* g=f^*\omega_N\alpha$. Indeed,
    if $\alpha^* f=g^*\omega_N\alpha$, then
    $\alpha^*\omega_N^*g^{**}=f^*\alpha^{**}$. Therefore,
    $\alpha^*g=\alpha^*\omega_N^*\omega_{N^*}g=\alpha^*\omega_N^*g^{**}\omega_M=f^*\alpha^{**}\omega_M=
    f^*\omega_N\alpha$, as required (we used the naturality of $\omega$ and the identity
    $\omega_N^*\omega_{N^*}=\id_{N^*}$ in the computation). The other direction follows by symmetry.
\end{proof}

\begin{thm}\label{TH:equivalence-of-categories}
    Let $\catC$ be a hermitian category.
    Define a functor $F:\Sesq{\catC}\to\Herm{\TDA{\catC}}$ by
    \[
    F(M,s)=((M,M, s^*\omega_M,s),(\id_M,\id_M^\op)),
    \]
    \[
    F(\psi)=(\psi,(\psi^{-1})^\op)
    \]
    for all $(M,s)\in\Sesq{\catC}$ and any morphism $\psi$ in $\Sesq{\catC}$.
    Then $F$ induces an equivalence of categories between $\Sesq{\catC}$ and $\Herm{\TDA{\catC}}$.
\end{thm}

\begin{proof}
    Let $(M,s)\in\Sesq{\catC}$. That
    $F(M,s)$ does lie in $\Herm{\TDA{\catC}}$
    follows from Proposition~\ref{PR:sesquilinear-forms-in-the-category-of-double-arrows}.
    Let $\psi:(M,s)\to (M',s')$ be an isometry. Then
    \[
    F(\psi)^*(\id_{M'},\id_{M'}^\op)F(\psi)=
    (\psi,(\psi^{-1})^\op)^*(\id_{M'},\id_{M'}^\op)(\psi,(\psi^{-1})^\op)=
    \]
    \[
    (\psi^{-1},\psi^\op)(\id_{M'},\id_{M'}^\op)(\psi,(\psi^{-1})^\op)
    =(\psi^{-1}\id_{M'}\psi,(\psi^{-1}\id_{M'}\psi)^\op) =
    (\id_M,\id_M^\op)\ .
    \]
    Thus, $F(\psi)$ is an isometry from $F(M,s)$ to $F(M',s')$. It is clear that
    $F$ respects composition, so we conclude that $F$ is a functor.

    To see that $F$ induces an equivalence, we construct a
    functor $G$ such that $F$ and $G$ are mutual inverses.
    Let $G:\Herm{\TDA{\catC}}\to\Sesq{\catC}$ be defined by
    \[
    G((M,N,f,g),(\alpha,\alpha^\op))=(M,\alpha^* g)
    \]
    \[
    G(\phi,\psi^\op)=\phi
    \]
    for all $((M,N,f,g),(\alpha,\alpha^\op))\in\Herm{\TDA{\catC}}$ and any morphism
    $(\phi,\psi^\op)$ in\linebreak $\Herm{\TDA{\catC}}$.

    Let $(Z,(\alpha,\alpha^\op)),(Z',(\alpha',\alpha'^\op))\in\Herm{\TDA{\catC}}$ and let
    $(\phi,\psi^\op):(Z,(\alpha,\alpha^\op))\to(Z',(\alpha',\alpha'^\op))$. It is easy to see
    that $G(Z,(\alpha,\alpha^\op))\in\Sesq{\catC}$, so we turn to check that $G(\phi,\psi^\op)$
    is an isometry from $G(Z,(\alpha,\alpha^\op))$ to $G(Z',(\alpha',\alpha'^\op))$.
    Writing $Z=(M,N,f,g)$ and $Z'=(M',N',f',g')$, this amounts
    to showing $\alpha^* g=\phi^*\alpha'^*g'\phi$.
    Indeed, since $(\phi,\psi^\op)$ is morphism from $Z$ to $Z'$, we have $g'\phi=\psi^*g$, and since
    $(\phi,\psi^\op)$ is an isometry, we also have $(\phi,\psi^\op)^*(\alpha',\alpha'^\op)(\phi,\psi^\op)=(\alpha,\alpha^\op)$,
    which in turn implies $\psi\alpha'\phi=\alpha$. We now have
    $\phi^*\alpha'^*g'\phi=\phi^*\alpha'^*\psi^*g=(\psi\alpha'\phi)^*g=\alpha^* g$, as required.
    That $G$ preserves composition is straightforward.

    It is easy to see that $GF$ is the identity functor on $\Sesq{\catC}$, so it is left to show
    that there is a natural isomorphism from $FG$ to $\id_{\Herm{\TDA{\catC}}}$.
    Keeping
    the notation of the previous paragraph, we have
    \[
    FG((M,N,f,g),(\alpha,\alpha^\op))=((M,M,(\alpha^*g)^*\omega_M,\alpha^*g),(\id_M,\id_M^\op))\ .
    \]
    By Proposition~\ref{PR:sesquilinear-forms-in-the-category-of-double-arrows}, we have
    $\alpha^* f=g^*\omega_N\alpha$, hence $(\alpha^*g)^*\omega_M=g^*\alpha^{**}\omega_M=g^*\omega_N\alpha=
    \alpha^*f$. Thus,
    \begin{equation}\label{EQ:computation-of-FG}
    FG((M,N,f,g),(\alpha,\alpha^\op))=((M,M,\alpha^*f,\alpha^*g),(\id_M,\id_M^\op))\ .
    \end{equation}
    Define a natural isomorphism $t:\id_{\Herm{\TDA{\catC}}}\to FG$ by
    $t_{(Z,(\alpha,\alpha^\op))}=(\id_M,\alpha^\op)$. Using \eqref{EQ:computation-of-FG}, it
    is easy to see that $t_{(Z,(\alpha,\alpha^\op))}$ is indeed an isometry from $(Z,(\alpha,\alpha^\op))$
    to $FG(Z,(\alpha,\alpha^\op))$. The map  $t$ is natural since for $Z',(\phi,\psi^\op)$ as above,
    we have $FG(\phi,\psi^\op)t_{(Z,(\alpha,\alpha^\op))}=(\phi,(\phi^{-1})^\op)(\id_M,\alpha^\op)=
    (\phi,(\alpha\phi^{-1})^\op)=(\phi,(\psi\alpha')^\op)=(\id_{M'},\alpha'^\op)(\phi,\psi^\op)=t_{(Z',(\alpha',\alpha'^\op))}(\phi,\psi^\op)$
    (we used the identity $\psi\alpha'\phi=\alpha$ verified above).
\end{proof}

\begin{remark}
    On the model of \cite[\S3]{BayerMoldovan12}, one can also  construct the \emph{category of
    (non-twisted) double arrows in $\catC$}, denoted $\DA{\catC}$.
    Its objects are quadruples $(M,N,f,g)$ with
    $M, N \in \catC$ and $f,g \in {\rm Hom}(M,N)$. A morphism from $(M,N,f,g)$ to $(M',N',f',g')$ is a pair $(\phi, \psi)$, where
    $\phi \in {\rm Hom}(M, M')$ and $\psi \in {\rm Hom}(N,N')$ satisfy $\psi f=f' \phi$ and $\psi g=g' \phi$.
    The category $\DA{\catC}$ is obviously additive and moreover, it
    admits a hermitian structure given by  $(M,N,f,g)^*=(N^*, M^*, g^*, f^*)$, $(\phi,\psi)^*=(\psi^*,\phi^*)$
    and $\omega_{(M,N,f,g)}=(\omega_M, \omega_N)$.

    There is a functor $T:\TDA{\catC}\to\DA{\catC}$ given by $T(M,N,f,g)= (M,N^*,f,g)$ and $T(\phi,\psi^\op)=(\phi,\psi^*)$.
    This functor induces an equivalence if $\catC$ is reflexive, but otherwise it need neither be faithful nor full.
    In addition, provided $\catC$ is reflexive, one can define a functor $F':\Sesq{\catC}\to\Herm{\DA{\catC}}$ by
    $F'(M,s)=((M,M^*,s^*\omega_M,s),(\omega_M,\id_{M^*}))$ and $F'(\psi)=(\psi,(\psi^{-1})^*)$. This  functor induces
    an equivalence of categories; the proof is analogous to \cite[Th.\ 4.1]{BayerMoldovan12}.
\end{remark}

\subsection{Hyperbolic Sesquilinear Forms}

Let $\catC$ be a hermitian category. The equivalence $\Sesq{\catC}\sim \Herm{\TDA{\catC}}$ of
Theorem~\ref{TH:equivalence-of-categories} allows us to pull back notions defined for unimodular hermitian forms
over $\TDA{\catC}$
to sesquilinear form over $\catC$. In this subsection, we will do this for hyperbolicity and thus obtain a notion
of a Witt group of sesquilinear forms.

Throughout, $F$ denotes the functor $\Sesq{\catC}\to \Herm{\TDA{\catC}}$ defined in Theorem~\ref{TH:equivalence-of-categories}.

\begin{dfn}\label{DF:hyperbolic-forms}
    A sesquilinear form $(M,s)$ over $\catC$ is called \emph{hyperbolic} if $F(M,s)$
    is hyperbolic as unimodular hermitian form over $\TDA{\catC}$.
\end{dfn}

The following proposition gives a more concrete meaning to hyperbolicity of sesquilinear forms over $\catC$.

\begin{prp}\label{PR:characterization-of-hyperbolic-forms-in-herm-categories}
    Up to isometry,
    the hyperbolic sesquilinear forms over $\catC$ are given by
    \[
    (M\oplus N,\smallSMatII{0}{f}{g}{0})
    \]
    where $M,N\in\catC$, $f\in\Hom_{\catC}(N,M^*)$, $g\in\Hom_{\catC}(M,N^*)$ and $\smallSMatII{0}{f}{g}{0}$
    is an element of $\Hom_{\catC}(M\oplus N,M^*\oplus N^*)$
    given in matrix form. Furthermore, a unimodular $\epsilon$-hermitian form is hyperbolic
    as a sesquilinear form
    (i.e.\ in the sense of Definition~\ref{DF:hyperbolic-forms}) if and only if it is hyperbolic as a unimodular
    $\epsilon$-hermitian form (see section \ref{section:hermitian-categories}).
\end{prp}

\begin{proof}
    Let $G$ be the functor $\Herm{\TDA{\catC}}\to\Sesq{\catC}$ defined in the proof
    of Theorem~\ref{TH:equivalence-of-categories}. Since $F$ and $G$ are mutual inverses, the hyperbolic
    sesquilinear forms over $\catC$ are the forms isometric to $G(Z\oplus Z^*,\Hyp{Z})$ for $Z\in\TDA{\catC}$.
    Write $Z=(M,N,h,g)$. Then
    \[
    (Z\oplus Z^*,\Hyp{Z})=\left(
    (M\oplus N,N\oplus M,\smallSMatII{h}{0}{0}{g^*\omega_N},\smallSMatII{g}{0}{0}{h^*\omega_N}),\smallSMatII{0}{\id_{Z^*}}{\omega_Z}{0}\right).
    \]
    Observe that $\smallSMatII{0}{\id_{Z^*}}{\omega_Z}{0}=\smallSMatII{0}{(\id_N,\id_M^\op)}{(\id_M,\id_N^\op)}{0}=
    (\smallSMatII{0}{\id_N}{\id_M}{0},\smallSMatII{0}{\id_N}{\id_M}{0}^\op)$.
    Thus,
    \[
    G(Z\oplus Z^*,\Hyp{Z})=(M\oplus N,\smallSMatII{0}{\id_N}{\id_M}{0}^*\smallSMatII{g}{0}{0}{h^*\omega_N}),
    \]
    and since $\smallSMatII{0}{\id_N}{\id_M}{0}^*\smallSMatII{g}{0}{0}{h^*\omega_N}
    =\smallSMatII{0}{\id_{M^*}}{\id_{N^*}}{0}\smallSMatII{g}{0}{0}{h^*\omega_N}=
    \smallSMatII{0}{h^*\omega_N}{g}{0}$, we see that $G(Z\oplus Z^*,\Hyp{Z})$ matches  the description
    in the proposition. Furthermore, by putting $h=f^*\omega_M$ for $f\in\Hom_{\catC}(N,M^*)$, we
    get $h^*\omega_N=\omega_M^* f^{**}\omega_N=\omega_M^*\omega_{M^*}f=f$.
    Thus,
    $(M\oplus N,\smallSMatII{0}{f}{g}{0})$ is hyperbolic for all $M,N,f,g$, as required.

    To finish, note that we have clearly shown that $(Q\oplus Q^*,\Hyp[\epsilon]{Q})$ is hyperbolic as a sesquilinear
    form for every $Q\in\catC$.
    To see the converse, assume $(M\oplus N,\smallSMatII{0}{f}{g}{0})$ is $\epsilon$-hermitian
    and unimodular.
    Then
    \[
    \smallSMatII{0}{f}{g}{0}=\epsilon \smallSMatII{0}{f}{g}{0}^*\omega_{M\oplus N}=
    \epsilon\smallSMatII{0}{g^*}{f^*}{0} \smallSMatII{\omega_M}{0}{0}{\omega_N}=\smallSMatII{0}{\epsilon g^*\omega_N}{\epsilon f^*\omega_M}{0},
    \]
    hence $g=\epsilon f^*\omega_N$ and $f=\epsilon g^*\omega_M$. Since $\smallSMatII{0}{f}{g}{0}$ is unimodular,
    $f$ and $g$ are bijective and hence, so are $\omega_N$ and $\omega_M$. In particular, $M$ is reflexive. It is now routine
    to verify that the map $\id_M\oplus f:M\oplus N\to M\oplus M^*$ is an isometry
    from $(M\oplus N,\smallSMatII{0}{f}{g}{0})$ to $(M\oplus M^*, \Hyp[\epsilon]{M})$, so
    the former is hyperbolic in the sense of section \ref{section:hermitian-categories}.
\end{proof}

Let $(A, \sigma)$ be a ring with involution.
In case $\catC$ is the category of
right $A$-modules, considered as
a hermitian category as in Example~\ref{EX:hermitian-category-obtained-from-a-ring-with-involution}, we obtain
a notion of hyperbolic sesquilinear forms over $(A,\sigma)$. These
hyperbolic forms can be characterized as follows.

\begin{prp}\label{PR:characterization-of-hyperbolic-forms-over-rings}
    A sesquilinear form $(M,s)$ over $(A,\sigma)$ is hyperbolic if and only if there
    are submodules $M_1,M_2\leq M$ such that $s(M_1,M_1)=s(M_2,M_2)=0$ and $M=M_1\oplus M_2$.
    Furthermore, if $(M,s)$ is unimodular and $\epsilon$-hermitian, then $(M,s)$ is hyperbolic as a sesquilinear
    space if and only if it is hyperbolic as an $\epsilon$-hermitian unimodular space.
\end{prp}

\begin{proof}
    Recall that for any two right $A$-modules $M_1,M_2$, we identify $(M_1\oplus M_2)^*$ with $M_1^*\oplus M_2^*$
    via $f\leftrightarrow (f|_{M_1},f|_{M_2})$. Let $(M,s)$ be a sesquilinear space
    and assume $M=M_1\oplus M_2$. By straightforward computation, we see that
    $\rAd{s}$ is of the form $\smallSMatII{0}{f}{g}{0}\in \Hom_A(M,M^*)=\Hom_A(M_1\oplus M_2,M_1^*\oplus M_2^*)$
    if and only if $s(M_1,M_1)=s(M_2,M_2)=0$. The proposition therefore follows from Proposition~\ref{PR:characterization-of-hyperbolic-forms-in-herm-categories}.
\end{proof}

\subsection{Witt Groups of Sesquilinear Forms}

Let $\catC$ be a hermitian category. Denote by $\WGrS{\catC}$
the Grothendieck group of isometry classes of \emph{sesquilinear} forms over $\catC$, with respect to orthogonal sum.
It is easy to see that the hyperbolic isometry classes span a subgroup
of $\WGrS{\catC}$, which we denote by $\mathbb{H}(\catC)$. The
\emph{Witt group of sesquilinear forms} over $\catC$ is defined to the quotient
\[
\WittS{\catC}=\WGrS{\catC}/\mathbb{H}(\catC)\ .
\]
By definition, we have $\WittS{\catC}\cong \Witt{\TDA{\catC}}$.
Taking $\catC$ to be the category of all (resp.\ reflexive, projective) right $A$-modules of finite
type and their duals, we obtain a notion of a Witt group for sesquilinear forms over $(A, \sigma)$.
Also observe that there is a homomorphism of groups
$\Witt[\epsilon]{\catC}\to\WittS{\catC}$ given by sending the class of a unimodular $\epsilon$-hermitian form
to its corresponding class in $\WittS{\catC}$. Corollary~\ref{CR:the-map-Witt-to-WittS-is-injective} below
presents sufficient conditions for the injectivity of this homomorphism.

\subsection{Extension of Scalars}
\label{subsection:F-commutes-with-scalar-ext}

Let $R/K$ be a commutative ring extension and let $\catC$ be a $K$-linear
hermitian category. Then the category $\TDA{\catC}$
is also $K$-linear.
For later usage, we now check that the scalar extension functor $\scalarExt{R}{K}$ of
subsection \ref{subsection:ring-extension} ``commutes'' with the
functor $F$ of Theorem~\ref{TH:equivalence-of-categories}.

\begin{prp}\label{PR:F-commutes-with-scalar-extension}
    There is a $1$-hermitian duality preserving functor
    $J:\TDA{\catC}_R\to\TDA{\catC_R}$ given by
    \[
    J((M,N,f,g)_R)=(M_R,N_R,f_R,g_R),
    \]
    \[
    J((\phi,\psi^\op)\otimes a)=(\phi\otimes a,(\psi\otimes a)^\op),
    \]
    for all $(M,N,f,g)\in\TDA{\catC}$ and any morphism $(\phi,\psi^\op)$ in $\TDA{\catC}$.
    (The associated natural isomorphism $i:J*\to *J$ is the identity map.)
    The functor $J$ is faithful and full, and it makes the following
    diagram commute:
    \[
    \xymatrix{
    \Sesq{\catC} \ar[rr]^F \ar[d]^{\SesqF{\scalarExt{R}{K}}} & & \Herm{\TDA{\catC}} \ar[d]^{\SesqF{\scalarExt{R}{K}}} \\
    \Sesq{\catC_R} \ar[r]^-F & \Herm{\TDA{\catC_R}} & \Herm{\TDA{\catC}_R} \ar[l]_-{\SesqF{J}}
    }
    \]
\end{prp}

\begin{proof}
    We only check that $J$ is faithful and full. All other assertions follow by computation.
    Let $Z,Z'\in\TDA{\catC}$.
    Define $I:\Hom_{\TDA{\catC_R}}(JZ_R,JZ'_R)\to\Hom_{\TDA{\catC}_R}(Z_R,Z'_R)$
    by
    \[
    I(\sum_i f_i\otimes a_i,(\sum_j g_j\otimes b_j)^\op)=\sum_{i,j}\left((f_i,0^\op)\otimes a_i+(0,g_j^\op)\otimes b_j\right)\ .
    \]
    Then it is routine to verify that $I$ is an inverse of
    $J:\Hom_{\TDA{\catC}_R}(Z_R,Z'_R)\to\Hom_{\TDA{\catC_R}}(JZ_R,JZ'_R)$. Thus, $J$ is full and faithful.
\end{proof}

As an immediate corollary, we get:

\begin{cor}\label{CR:F-preserves-isometry-after-scalar-extension}
    Let $(M,s)$, $(M',s')$ be two sesquilinear forms over $\catC$. Then
    $\SesqF{\scalarExt{R}{K}}(M,s)$ is isometric to $\SesqF{\scalarExt{R}{K}}(M',s')$ if and only if
    $\SesqF{\scalarExt{R}{K}} F(M,s)$ is isometric to $\SesqF{\scalarExt{R}{K}} F(M',s')$.
\end{cor}

\section{Systems of Sesquilinear Forms}
\label{section:systems}

In this section, we explain how to generalize the results of  Section~\ref{section:equivalence-of-categories}
to systems of sesquilinear forms.

\medskip

Let $A$ be a ring and let $\{\sigma_i\}_{i\in I}$ be an nonempty family of (not necessarily
distinct) involutions of $A$.
A system of sesquilinear forms over $(A,\{\sigma_i\}_{i\in I})$ is a pair $(M,\{s_i\}_{i\in I})$
such that $(M,s_i)$ is a sesquilinear space over $(A,\sigma_i)$ for all $i$.
An isometry between two systems of sesquilinear forms $(M,\{s_i\}_{i\in I})$, $(M',\{s'_i\}_{i\in I})$ is
an isomorphism $f:M\to M'$ such that $s'_i(fx,fy)=s_i(x,y)$ for all $x,y\in M$, $i\in I$.

Observe that each of the involutions $\sigma_i$ gives rise to a hermitian structure $(*_i,\omega_i)$
on $\rMod{A}$, the category of right $A$-modules. In particular, a system of sesquilinear forms
$(M,\{s_i\})$ gives rise to homomorphisms $\rAd{(s_i)},\lAd{(s_i)}:M\to M^{*_i}$ given
by $\rAd{(s_i)}(x)(y)=\sigma_i(s_i(y,x))$ and $\lAd{(s_i)}(x)(y)=s_i(x,y)$, where $M^{*_i}=\Hom_A(M,A)$, considered
as a right $A$-module via the action $(f\cdot a)m=\sigma_i(a)f(m)$.
This leads to the notion of systems of sesquilinear forms over hermitian categories.

\medskip

Let $\catC$ be an additive category and let $\{*_i,\omega_i\}_{i\in I}$ be a nonempty
family of hermitian structures on $\catC$. A system of sesquilinear forms over $(\catC,\{*_i,\omega_i\}_{i\in I})$
is a pair $(M,\{s_i\}_{i\in I})$ such that $M\in\catC$ and $(M,s_i)$ is a sesquilinear form
over $(\catC,*_i,\omega_i)$.
An isometry
between two systems of sesquilinear forms $(M,\{s_i\}_{i\in I})$ and $(M',\{s'_i\}_{i\in I})$ is
an isomorphism $f:M\xrightarrow{\sim} M'$ such that $f^{*_i}s'_i f=s_i$ for all $i\in I$.
We let $\SysSesq{\catC}{I}$ (or $\SysSesq{\catC,\{*_i,\omega_i\}}{I}$)
denote the category of systems of sesquilinear forms over $(\catC,\{*_i,\omega_i\}_{i\in I})$
with isometries as morphisms.

\medskip

Keeping the notation of the previous paragraph,
the results of section \ref{section:equivalence-of-categories} can be extended to
systems of sesquilinear forms as follows: Define the category of \emph{twisted
double $I$-arrows} over $(\catC,\{*_i,\omega_i\}_{i\in I})$, denoted $\TDA[I]{\catC}$, to be the category
whose objects are quadruples $(M,N,\{f_i\}_{i\in I},\{g_i\}_{i\in I})$ with $M,N\in\catC$ and $f_i,g_i\in \Hom_{\catC}(M,N^{*_i})$.
A morphism $(M,N,\{f_i\},\{g_i\})\to(M',N',\{f'_i\},\{g'_i\})$ is a formal pair
$(\phi,\psi^\op)$ such that $\phi\in\Hom(M,M')$, $\psi\in\Hom(N',N)$ and $\psi^{*_i}f_i=f'_i\phi$, $\psi^{*_i}g_i=g'_i\phi$
for all $i\in I$. The composition is defined by the formula $(\phi,\psi^\op)(\phi',\psi'^\op)=(\phi\phi',(\psi'\psi)^\op)$.

The category $\TDA[I]{\catC}$ can be made into a reflexive hermitian category  by letting
$(M,N,\{f_i\},\{g_i\})^*=(N,M,\{g_i^{*_i}\omega_{i,N}\},\{f_i^{*_i}\omega_{i,M}\})$, $(\phi,\psi^\op)^*=(\psi,\phi^\op)$
and $\omega_{(M,N,\{f_i\},\{g_i\})}=(\id_M,\id_N^\op)$.
It is now possible to prove the following theorem, whose proof is completely analogous to the proof
of Theorem~\ref{TH:equivalence-of-categories}.

\begin{thm}\label{TH:equivalence-of-categories-for-systems}
    Define a functor $F:\SysSesq{\catC}{I}\to\Herm{\TDA[I]{\catC}}$ by
    \[
    F(M,\{s_i\})=((M,M,\{s_i^{*_i}\omega_{i,M}\},\{s_i\}),(\id_M,\id_M^\op)),
    \]
    \[
    F(\psi)=(\psi,(\psi^{-1})^\op)
    \]
    Then $F$ induces an equivalence of categories.
\end{thm}

\begin{proof}[Proof (sketch)]
    It is easy to see
    that any hermitian form over $\Herm{\TDA[I]{\catC}}$ has the form
    $((M,N,\{f_i\},\{g_i\}),(\alpha,\alpha^\op))$. Define
    a functor $G:\Herm{\TDA[I]{\catC}}\to\SysSesq{\catC}{I}$ by
    \[
    G((M,N,\{f_i\},\{g_i\}),(\alpha,\alpha^\op))=(M,\{\alpha^{*_i}g_i\}),
    \]
    \[
    G(\phi,\psi^\op)=\phi\ .
    \]
    By arguing as in the proof of Theorem~\ref{TH:equivalence-of-categories},
    we see that $F$ and $G$ are mutual inverses.
\end{proof}

As we did in section~\ref{section:equivalence-of-categories}, we can use
Theorem~\ref{TH:equivalence-of-categories-for-systems} to define hyperbolic systems of sesquilinear
forms. Namely, a system of forms $(M,\{s_i\})$ over $\catC$ will be called \emph{hyperbolic}
if $F(M,\{s_i\})$ is hyperbolic over $\TDA[I]{\catC}$. The following two propositions
are proved in the same manner as Propositions~\ref{PR:characterization-of-hyperbolic-forms-in-herm-categories} and
\ref{PR:characterization-of-hyperbolic-forms-over-rings}, respectively.

\begin{prp}
    A system of sesquilinear forms $(M,\{s_i\})$ over $\catC$ is hyperbolic
    if and only if there are $M_1,M_2\in\catC$, $f_i\in \Hom(M_2,M_1^{*_i})$,
    $g_i\in\Hom(M_1,M_2^{*_i})$ such that $M= M_1\oplus M_2$ and for all $i\in I$,
    \[
    s_i=\smallSMatII{0}{f_i}{g_i}{0}\in\Hom(M,M^{*_i})=\Hom(M_1\oplus M_2,M_1^{*_i}\oplus M_2^{*_i})\ .
    \]
    In this case, each of the sesquilinear forms $(M,s_i)$ (over $(\catC,*_i,\omega_i)$)
    is hyperbolic.
\end{prp}

\begin{prp}
    Let $A$ be a ring and let $\{\sigma_i\}_{i\in I}$ be a nonempty family of involutions of $A$.
    A system of sesquilinear forms $(M,\{s_i\})$ over $(A,\{\sigma_i\})$ is hyperbolic
    if and only if there are submodules $M_1,M_2\leq M$ such that $M=M_1\oplus M_2$
    and $s_i(M_1,M_1)=s_i(M_2,M_2)=0$ for all $i\in I$. In this case, each of the sesquilinear
    forms $(M,s_i)$ (over $(A,\sigma_i)$) is hyperbolic.
\end{prp}

The notion of hyperbolic systems of sesquilinear forms can be used to define Witt groups.
We leave the details to the reader.

\medskip

Let $R/K$ be a commutative ring extension. If $\catC$ and all the hermitian
structures $\{*_i,\omega_i\}_{i\in I}$ are  $K$-linear, then the scalar extension functor
$\scalarExt{R}{K}:\catC\to\catC_R$ is $1$-hermitian and duality preserving with respect to $(*_i,\omega_i)$
for all $i\in I$. Therefore, we have a functor $\SesqF{\scalarExt{R}{K}}:\SysSesq{\catC}{I}\to\SysSesq{\catC_R}{I}$
given by $\SesqF{\scalarExt{R}{K}}(M,\{s_i\}_{i\in I})=(M_R,\{(s_i)_R\}_{i\in I})$.
We thus have a notion of scalar extension for systems of bilinear forms (and it agrees
with the obvious scalar extension for systems of bilinear forms over a ring with a family of involutions,
provided the assumptions of Remark~\ref{RM:scalar-extension-in-module-categories} hold).
Using the ideas of subsection~\ref{subsection:F-commutes-with-scalar-ext}, one can show:

\begin{cor}\label{CR:F-for-systems-preserves-isometry-after-scalar-extension}
    Let $(M,\{s_i\})$, $(M',\{s'_i\})$ be two systems of sesquilinear forms over $(\catC,\{*_i,\omega_i\})$. Then
    $\SesqF{\scalarExt{R}{K}}(M,\{s_i\})$ is isometric to $\SesqF{\scalarExt{R}{K}}(M',\{s'_i\})$ if and only if
    $\SesqF{\scalarExt{R}{K}} F(M,\{s_i\})$ is isometric to $\SesqF{\scalarExt{R}{K}} F(M',\{s'_i\})$.
\end{cor}

\section{Applications}
\label{section:applications}

The following section uses the previous results to generalize various known results about hermitian forms (over
rings or reflexive hermitian categories) to systems of sesqui\-li\-near forms
over (not-necessarily reflexive) hermitian categories. Some of the consequences to follow
were obtained in \cite{BayerMoldovan12}
for hermitian forms over rings. Here we rephrase them for hermitian categories, extend them to systems of
sesquilinear forms and drop the assumption that the base module (or object) is reflexive.

\subsection{Witt's Cancelation Theorem}
\label{subsection:cancelation}

Quebbemann, Scharlau and Schulte (\cite[\S3.4]{QuSchSch79}) have proven Witt's Cancelation Theorem for unimodular
hermitian forms over hermitian categories $\catC$ satisfying the following conditions:
\begin{enumerate}
    \item[(a)] All idempotents in $\catC$ split (see subsection~\ref{subsection:transfer}).
    \item[(b)] For all $C\in\catC$, $E:=\End_{\catC}(C)$ is a \emph{complete semilocal} ring in which $2$ is invertible.
\end{enumerate}
Recall that complete semilocal means that $E/\Jac(E)$ is semisimple (i.e.\ $E$ is semilocal)
and that the standard map $E\to \invlim \{E/\Jac(R)^n\}_{n\in\N}$ is an isomorphism (i.e.\ $E$ is complete in the \emph{$\Jac(E)$-adic topology}).
In fact, condition (a) can be dropped since idempotents can be split artificially (see subsection~\ref{subsection:odd-degree} below),
or alternatively, since by applying transfer (see subsection~\ref{subsection:transfer})   one can move to a module category in which idempotents split.

We shall now use the Quebbemann-Scharlau-Schulte cancelation theorem together with Theorem~\ref{TH:equivalence-of-categories-for-systems}
to give several conditions guaranteeing cancelation for systems of sesquilinear forms.

\medskip

Our first criterion is based on the following well-known lemma.

\begin{lem}\label{LM:complete-semilocal-K-algs}
    Let $K$ be a commutative noetherian complete semilocal ring (e.g.\ a complete
    discrete valuation ring).
    Then any $K$-algebra $A$ which is finitely generated as a $K$-module is complete semilocal.
\end{lem}

\begin{proof}
    For brevity, write $I=\Jac(K)$ and $J=\Jac(A)$.
    By \cite[Th.\ 2]{Hi60} and the proof of \cite[Pr.\ 8.8(i)]{Fi12} (for instance), $A=\invlim \{A/A(I^n)\}_{n\in \N}$.
    That $A=\invlim\{A/J^n\}_{n\in\N}$ would follow if we verify that $J^m\subseteq AI\subseteq J$ for some $m\in\N$.
    The right inclusion holds since $1+AI$ consists of right invertible elements. Indeed, for all $a\in AI$,
    we have $aA+AI=A$, so by Nakayama's Lemma (applied to the $K$-module $A$), $aA=A$.
    The existence of $m$, as well as the fact that $A$ is semilocal, follows by arguing as in \cite[Ex.~2.7.19'(ii)]{Ro88}
    (for instance).
\end{proof}

\begin{thm}\label{TH:our-cancelation-I}
    Let  $K$ be a commutative noetherian complete semilocal ring with $2\in \units{K}$, let
    $\catC$ be a $K$-category equipped with
    $K$-linear hermitian structures  $\{*_i,\omega_i\}_{i\in I}$,
    and let
    $(M,\{s_i\})$, $(M',\{s'_i\})$, $(M'',\{s''_i\})$
    be systems of sesquilinear forms over $(\catC,\{*_i,\omega_i\})$.
    Assume that $\Hom_{\catC}(M,N)$ is finitely generated as a $K$-module for all $M,N\in\catC$.
    Then
    \[
    (M,\{s_i\})\oplus (M',\{s'_i\})\simeq (M,\{s_i\})\oplus (M'',\{s''_i\})
    ~\iff~
    (M',\{s'_i\})\simeq (M'',\{s''_i\})
    .
    \]
\end{thm}

\begin{proof}
    In light of Theorem~\ref{TH:equivalence-of-categories-for-systems}, it is enough to prove
    cancelation of unimodular $1$-hermitian  forms over the the category
    $\TDA[I]{\catC}$ (note that the equivalence of Theorem~\ref{TH:equivalence-of-categories-for-systems} respects orthogonal sums).
    This would follow from the cancelation theorem of \cite[\S3.4]{QuSchSch79} if we show that the endomorphism
    rings of objects in $\TDA[I]{\catC}$ are complete semilocal rings in which $2$ is invertible.
    Indeed, let $Z:=(M,N,\{f_i\},\{g_i\})\in\TDA[I]{\catC}$. Then $E:=\End(Z)$ is a subring of
    $\End_{\catC}(M)\times\End_{\catC}(N)^\op$, which is a $K$-algebra by assumption.
    Since the hermitian structures $\{*_i,\omega_i\}$ are $K$-linear,
    $E$ is in fact a $K$-subalgebra, which must be f.g.\ as a $K$-module (because this is true
    for $\End_{\catC}(M)\times\End_{\catC}(N)^\op$ and $K$ is noetherian). Thus, we are done by Lemma~\ref{LM:complete-semilocal-K-algs}
    and the fact that $2\in\units{K}$.
\end{proof}

As corollary, we get the following result which resembles \cite[Th.~8.1]{BayerMoldovan12}.

\begin{cor}
    Let $K$ be a commutative noetherian complete semilocal ring with $2\in \units{K}$, let $A$ be a $K$-algebra
    which is finitely generated as a $K$-module, and let $\{\sigma_i\}_{i\in I}$ be a family of $K$-involutions
    on $A$. Then cancelation holds for systems of sesquilinear forms over $(A,\{\sigma_i\})$ which are defined
    on finitely generated right $A$-modules.
\end{cor}

For the next theorem, recall that a ring $R$ is said to be \emph{semiprimary}
if $R$ is semilocal and $\Jac(R)$ is nilpotent. For example, all artinian rings are semiprimary.
Note that all semiprimary rings are complete semilocal.
It is well-known that for a ring $R$ and an idempotent $e\in R$,
$R$ is semiprimary if and only if $eRe$ and $(1-e)R(1-e)$ are semiprimary. As a result, if $M,N$
are two objects in an additive category, then $\End(M\oplus N)$ is semiprimary if and only if $\End(M)$
and $\End(N)$ are semiprimary.

\begin{thm}\label{TH:our-cancelation-II}
    Let $\catC$ be an additive category with hermitian structures $\{*_i,\omega_i\}$
    and
    let
    $(M,\{s_i\})$, $(M',\{s'_i\})$, $(M'',\{s''_i\})$
    be systems of sesquilinear forms over $(\catC,\{*_i,\omega_i\})$.
    Assume that $\End_{\catC}(M)$, $\End_{\catC}(M')$, $\End_{\catC}(M'')$ are semiprimary rings in which
    $2$ is invertible.
    Then
    \[
    (M,\{s_i\})\oplus (M',\{s'_i\})\simeq (M,\{s_i\})\oplus (M'',\{s''_i\})
    ~\iff~
    (M',\{s'_i\})\simeq (M'',\{s''_i\})
    .
    \]
\end{thm}

\begin{proof}
    As in the proof of Theorem~\ref{TH:our-cancelation-I},  it is enough to show that
    the objects in $\TDA[I]{\catC}$ have a complete semilocal endomorphism ring.
    In fact, we may restrict to those objects $Z:=(M,N,\{f_i\},\{g_i\})$ for which $\End_{\catC}(M)$ and
    $\End_{\catC}(N)$ are semiprimary. (These do form a hermitian subcategory of $\TDA[I]{\catC}$
    by the comments above.) Fix such $Z$ and let $H=\bigoplus_{i\in I}\Hom_{\catC}(M,N^{*_i})$.
    We view the morphism $\{f_i\}$ and $\{g_i\}$ as elements of $H$ in the obvious way. Let
    $A=\End(M)$ and $B=\End(N)$.
    We endow $H$ with a $(B^\op, A)$-bimodule structure by setting
    $b^\op\circ(\bigoplus_{i\in I} h_i)\circ a=\bigoplus_{i\in I}(b^{*_i}\circ h_i\circ a)$ for
    all $a\in A$, $b\in B$, $\bigoplus_i h_i\in H$. This allows us to construct the ring
    $S:=\smallSMatII{A}{}{H}{B^\op}$. It is now straightforward to check
    that $\End(Z)$ consists of those elements in $A\times B^\op=\smallSMatII{A}{}{}{B^\op}$
    that commute with $\smallSMatII{0}{}{f_i}{0}$ and $\smallSMatII{0}{}{g_i}{0}$ for all $i\in I$.
    Thus, $\End(Z)$ is a \emph{semi-centralizer}
    subring of $A\times B^\op$ in the sense of \cite[\S1]{Fi12}. By \cite[Th.\ 4.6]{Fi12}, a semi-centralizer
    subring of a semiprimary ring is semiprimary, so $\End(Z)$ is semiprimary, and in particular
    complete semilocal.
\end{proof}

\begin{cor}
    Let $A$ be a semiprimary ring with $2\in\units{A}$, and let $\{\sigma_i\}_{i\in I}$ be a family of involutions
    on $A$. Then cancelation holds for
    systems of sesquilinear forms over $(A,\{\sigma_i\})$ which are defined
    on finitely presented right $A$-modules.
\end{cor}

\begin{proof}
    By \cite[Th.\ 4.1]{Bj71} (or \cite[Th.\ 7.3]{Fi12}), the endomorphism ring of a finitely presented $A$-module is semiprimary.
    Now apply Theorem~\ref{TH:our-cancelation-II}.
\end{proof}

\begin{cor}
    Let $\catC$ be an \emph{abelian} category equipped with hermitian structures $\{*_i,\omega_i\}$.
    Assume that $\catC$ consists of objects of finite length. Then cancelation holds for systems of sesquilinear
    forms over $(\catC,\{*_i,\omega_i\})$.
\end{cor}

\begin{proof}
    By the Hadara-Sai Lemma (\cite[Pr.\ 2.9.29]{Ro88}), the endomorphism ring of an object of finite
    length in an abelian category is semiprimary, so we are done by Theorem~\ref{TH:our-cancelation-II}.
    Alternatively, one can check directly that the category $\TDA[I]{\catC}$ is abelian and consists of objects of
    finite length, apply the Hadara-Sai Lemma to $\TDA[I]{\catC}$,
    and then use  the cancelation theorem of \cite[\S3.4]{QuSchSch79}.
\end{proof}

\begin{remark}
    It is not hard to deduce from a theorem of Camps and Dicks \cite[Cr.~2]{CaDi93} that
    if the endomorphism rings of $\catC$ are semilocal, then so are the endomorphism rings of
    $\TDA[I]{\catC}$. (Simply check  that $\End(M,N,\{f_i\},\{g_i\})$ is a
    \emph{rationally closed} subring of $\End_{\catC}(M)\times\End_{\catC}(N)^\op$ in the sense of \cite[p.\ 204]{CaDi93}.)
    By applying transfer (see subsection~\ref{subsection:transfer}) to $\TDA[I]{\catC}$, one can then move to
    the context of unimodular $1$-hermitian
    forms over semilocal rings. Cancelation theorems for such forms were obtained by various authors
    including Knebusch \cite{Kne69}, Reiter \cite{Reiter75} and
    Keller \cite{Kel88}. However, none of these apply to the general case, as in fact cancelation
    is no longer true; see \cite[\S2]{Kel88}.
    Nevertheless, the cancelation results of \cite{Kel88} can still be used to get some partial
    results about systems
    of sesquilinear forms over $\catC$; we leave the details to the reader.
\end{remark}

\remove{
\begin{remark}
    See \cite{HerSha95}, \cite{FacHer06} and related papers
    for further examples of modules with semilocal endomorphism
    ring.
\end{remark}
}

\subsection{Finiteness Results}

In the following two subsections, we generalize the finiteness results of \cite[\S10]{BayerMoldovan12}
to systems of sesquilinear forms.

\medskip

For a ring $A$, we denote by $T(A)$ the $\Z$-torsion subgroup of $A$.
Recall that if $R$ is a commutative ring,
$A$ is said to be \emph{$R$-finite} if $A_R=A \otimes_{\Z} R$ is a finitely generated $R$-module and $T(A)$ is finite.
Note that being $R$-finite passes to subrings.

The proofs of the results to follow are completely analogous to the proofs of the
corresponding statements in \cite[\S10]{BayerMoldovan12};
they are based on applying the equivalence of Theorem~\ref{TH:equivalence-of-categories-for-systems}
and then using the finiteness results of \cite{BayKeaWil89}, possibly after applying transfer.

Throughout,  $\catC$ is an additive category and $\{*_i,\omega_i\}_{i\in I}$ is
a nonempty family of hermitian structures on $\catC$.
Fix a system of sesquilinear forms $(V, \{s_i\}_{i \in I})$ over $(\catC, \{\*_i,\omega_i\})$
and let $Z(V,\{s_i\})=(V,V,\{s_i^{*_i}\omega_{i,V}\},\{\rAd{s_i}\})\in\TDA[I]{\catC}$.
(Note that $F(V,\{s_i\})=(Z,(\id_V,\id_V^\op))$
with $F$ as in Theorem~\ref{TH:equivalence-of-categories-for-systems}).

\begin{thm}
If there exists a non-zero integer $m$ such that $\End_{\catC}(V)$ is $\Z[1/m]$-finite, then there
are  finitely many isometry classes of summands of $(V,\{s_i\})$.
\end{thm}

\begin{thm}
Assume that there is a non-zero integer $m$ such that the ring $\End_{\TDA[I]{\catC}}(Z(V,\{s_i\}))$ is $\Z[1/m]$-finite
(e.g.\ if $\End_{\catC}(V)$ is $\Z[1/m]$-finite).
Then there exist only finitely many isometry classes of systems of sesquilinear forms
$(V', \{s'_i\}_{i \in I})$ over $\catC$ such that $Z(V', \{s'_i\}) \simeq Z(V,\{s_i\})$
(as objects in $\TDA[I]{\catC}$).
\end{thm}

\subsection{Finiteness of The Genus}

Let $\catC$ be a hermitian category admitting a nonempty family of hermitian structures $\{*_i,\omega_i\}_{i\in I}$.
We say that two systems of sesquilinear forms $(M,\{s_i\})$, $(M',\{s'_i\})$ are \emph{of the same genus} if they
become isometric after applying $\scalarExt{\Z_p}{\Z}$ for every prime number $p$ (where $\Z_p$ are the $p$-adic integer).
(See Remark~\ref{RM:scalar-extension-in-module-categories} for conditions under which this definition of genus
agrees with the naive definition of genus for module categories.)
As in \cite[Th.\ 10.3]{BayerMoldovan12}, we have:

\begin{thm}
    Let $(M,\{s_i\})$ be a system of sesquilinear forms over $(\catC,\{*_i,\omega_i\})$, and assume
    that $\End(M)$ is $\Q$-finite. Then the genus of $(M,\{s_i\})$ contains only a finite number of isometry classes of
    systems of sesquilinear forms.
\end{thm}

\subsection{Forms That Are Trivial in The Witt Group}

Let $\catC$ be a hermitian category.
By definition, a
unimodular $\epsilon$-hermitian (resp.\ sesquilinear) form $(M,s)$
is trivial  in $\Witt[\epsilon]{\catC}$ (resp.\
$\WittS{\catC}$)
if and only if there are
unimodular $\epsilon$-hermitian (resp.\ sesquilinear) hyperbolic forms $(H_1,h_1)$, $(H_2,h_2)$
such that $(M,s)\oplus (H_1,h_1)\simeq (H_2,h_2)$. In this section, we will show
that under mild assumptions, this implies that $(M,s)$ is hyperbolic.

\begin{lem}\label{LM:only-on-hyp-form-on-an-object}
    Let $M\in\catC$, and assume that $M$ is
    a (finite) direct sum of objects with local endomorphism ring.
    Then, up to isometry, there is at most one $\epsilon$-hermitian hyperbolic form on $M$.
\end{lem}

\begin{proof}
    For $X\in\catC$, let $[X]$ denote the isomorphism class of $X$.
    The Krull-Schmidt Theorem (e.g.\ see \cite[pp.\ 237 ff.]{Ro88}) implies that if $M\cong \bigoplus_{i=1}^t M_i$
    with each $M_i$ indecomposable, then the unordered list $[M_1],\dots,[M_t]$ is determined by $M$.

    Let $(M,s)$ be an $\epsilon$-hermitian hyperbolic form on $M$, say $(M,s)\simeq(N\oplus N^*,\Hyp[\epsilon]{N})$.
    Write $N\cong \oplus_{i=1}^r N_i$ with each $N_i$ indecomposable.
    Then $s\simeq \bigoplus_{i=1}^r\Hyp[\epsilon]{N_i}$.
    It is easy to check that the isometry
    class of $\Hyp[\epsilon]{N_i}$ depends only on the set $\{[N_i],[N_i^*]\}$. Furthermore, using the Krull-Schmidt
    Theorem, one easily verifies that the unordered list  $\{[N_1],[N_1^*]\},\dots,\{[N_r],[N_r^*]\}$
    is uniquely determined by $M$. It follows that $(M,s)$ is isometric to a sesquilinear form which is determined by
    $M$ up to isometry.
\end{proof}

\begin{prp}\label{PR:characterization-of-differences-of-hyp-forms}
    Let $\catC$ be a hermitian category satisfying conditions (a), (b) of subsection~\ref{subsection:cancelation}.
    Then a unimodular $\epsilon$-hermitian form $(M,s)$
    is trivial in $\Witt[\epsilon]{\catC}$ if and only if it is hyperbolic.
\end{prp}

\begin{proof}
    Note first that conditions (a) and (b) imply that every object of $\catC$ is a sum of objects with local
    endomorphism ring, hence we may apply the Krull-Schmidt Theorem to $\catC$. (For example, this follows from
    \cite[Th.\ 2.8.40]{Ro88} since the endomorphism rings of $\catC$ are \emph{semiperfect}.)
    Let $(M,s)$ be a unimodular $\epsilon$-hermitian form such that $(M,s)\equiv 0$ in $\Witt[\epsilon]{\catC}$.
    Then there are unimodular $\epsilon$-hermitian  hyperbolic forms $(H_1,h_1)$, $(H_2,h_2)$
    such that $(M,s)\oplus (H_1,h_1)\simeq (H_2,h_2)$. Using the Krull-Schmidt Theorem,
    it is  easy to see that there is $N\in \catC$ such that $M\cong N\oplus N^*$.
    Thus, we may consider $\Hyp[\epsilon]{N}$ as a hermitian form on $M$.
    By Lemma~\ref{LM:only-on-hyp-form-on-an-object}, we have $\Hyp[\epsilon]{N}\oplus h_1\simeq h_2$, implying
    $\Hyp[\epsilon]{N}\oplus h_2\simeq s\oplus h_2$. Therefore,
    by the cancelation theorem of \cite[\S3.4]{QuSchSch79}, $s\simeq\Hyp[\epsilon]{N}$, as required.
\end{proof}

\begin{prp}\label{PR:characterization-of-differences-of-hyp-forms-II}
    Let $\catC$ be a hermitian category in which all idempotents split and such that either
    \begin{enumerate}
        \item[(1)] $\catC$ is $K$-linear, where $K$
        is a noetherian complete semilocal ring with $2\in\units{K}$, and all $\Hom$-sets in $\catC$
        are finitely generated as $K$-modules, or
        \item[(2)] for all $M\in\catC$, $\End_{\catC}(M)$ is semiprimary and $2\in\units{\End_{\catC}(M)}$.
    \end{enumerate}
    Then a sesquilinear form $(M,s)$
    is trivial in $\WittS{\catC}$ if and only if it is hyperbolic.
\end{prp}

\begin{proof}
    It is enough to verify that $F(M,s)$ is hyperbolic in $\TDA{\catC}$ (Theorem~\ref{TH:equivalence-of-categories}).
    The proofs of Theorems~\ref{TH:our-cancelation-I}
    and~\ref{TH:our-cancelation-II} imply that $\TDA{\catC}$ satisfies
    condition (b) of subsection~\ref{subsection:cancelation}, and condition (a) is routine (see also Lemma~\ref{LM:Karoubi}(ii) below).
    Therefore,  $F(M,s)$ is hyperbolic by Proposition~\ref{PR:characterization-of-differences-of-hyp-forms-II}.
\end{proof}


\begin{cor}\label{CR:the-map-Witt-to-WittS-is-injective}
    Under the assumptions of Proposition~\ref{PR:characterization-of-differences-of-hyp-forms-II},
    the map $\Witt{\catC}\to\WittS{\catC}$ is injective.
\end{cor}

\begin{proof}
    This follows from Propositions~\ref{PR:characterization-of-differences-of-hyp-forms-II}
    and~\ref{PR:characterization-of-hyperbolic-forms-in-herm-categories}.
\end{proof}

\subsection{Odd Degree Extensions}
\label{subsection:odd-degree}

Throughout this subsection, $L/K$ is an odd degree field extension and $\Char K\neq 2$.
A well known theorem of Springer asserts that two unimodular hermitian forms over $K$ become isometric over
$L$ if and only if they are already isometric over $K$. Moreover,
the restriction map (i.e.\ the scalar extension map) $r_{L/K}: \Witt{K} \rightarrow \Witt{L}$ is injective.
Both statements were extended to hermitian forms over finite dimensional $K$-algebras
with $K$-linear involution  in \cite[Pr.\ 1.2 and Th.\ 2.1]{BayLen90} (see also \cite{Fain} for
a version in which $L/K$ is replaced with an extension of complete discrete valuation rings).
In this section, we extend these results to sesquilinear forms over hermitian categories.

\begin{thm}\label{TH:our-Springer}
    Let $\catC$ be an additive $K$-category such that $\dim_K\Hom(M,M')$ is finite
    for all $M,M'\in\catC$. Let $\{*_i,\omega_i\}_{i\in I}$ be a nonempty family
    of $K$-linear hermitian structures
    on $\catC$ and let $(M,\{s_i\})$, $(M',\{s'_i\})$ be two systems of sesquilinear
    forms over $(\catC,\{*_i,\omega_i\})$. Then $\scalarExt{L}{K}(M,\{s_i\})\simeq \scalarExt{L}{K}(M',\{s'_i\})$
    if and only if $(M,\{s_i\})\simeq(M',\{s'_i\})$.
\end{thm}

\begin{proof}
    By Corollary~\ref{CR:F-for-systems-preserves-isometry-after-scalar-extension}, it is enough
    to prove $\scalarExt{L}{K}F(M,\{s_i\})\simeq \scalarExt{L}{K}F(M',\{s'_i\})$
    if and only if $F(M,\{s_i\})\simeq F(M',\{s'_i\})$ (with
    $F$ as in
    Theorem~\ref{TH:equivalence-of-categories-for-systems}).
    Write $(Z,(\alpha,\alpha^\op))=F(M,\{s_i\})\oplus F(M',\{s'_i\})$
    and let $E=\End(Z)$.
    Then $E$ is a $K$-subalgebra of $\End(M\oplus M')\times\End(M\oplus M')^\op$, which is finite dimensional.
    By applying $\Transfer{(Z,(\alpha,\alpha^\op))}$ (see subsection \ref{subsection:transfer}), we reduce to
    showing that two \mbox{$1$-hermitian} forms over
    $E$ are isometric over $E\otimes_K L$ if and only if they are isometric over $E$,
    which is just \cite[Th.\ 2.1]{BayLen90}. (Note that we used
    the fact that transfer commutes with $\scalarExt{L}{K}$ in the sense of
    subsection \ref{subsection:scalar-extension-commutes-with-transfer}.)
\end{proof}

\begin{cor}
    Let $A$ be a finite dimensional $K$-algebra and let $\{\sigma_i\}_{i\in I}$
    be a nonempty family of $K$-involutions on $A$. Let $(M,\{s_i\})$, $(M',\{s'_i\})$ be two systems of sesquilinear
    forms over $(A,\{\sigma_i\})$. If $M$ and $M'$ are of finite type, then $\scalarExt{L}{K}(M,\{s_i\})\simeq \scalarExt{L}{K}(M',\{s'_i\})$
    if and only if $(M,\{s_i\})\simeq(M',\{s'_i\})$.
\end{cor}

To state the analogue of the injectivity of
$r_{L/K}: \Witt{K} \rightarrow \Witt{L}$ for hermitian categories, we need to introduce additional notation.

\smallskip

An additive category $\catC$ is called \emph{pseudo-abelian} if all idempotents in $\catC$ split.
Any additive category $\catC$ admits a \emph{pseudo-abelian closure} (e.g.\ see \cite[Th.\ 6.10]{Kar78}), namely, a pseudo-abelian additive
category $\catC^\circ$ equipped with an additive functor $A\mapsto A^\circ:\catC\to \catC^\circ$, such that the pair $(\catC^\circ, A\mapsto A^\circ)$
is \emph{universal}. The category $\catC^\circ$ is unique up to equivalence and the functor $A\mapsto A^\circ$
turns out to be faithful and full.
The category $\catC^\circ$ can be realized as the category of pairs $(M,e)$ with $M\in\catC$ and $e\in\End_{\catC}(M)$ an idempotent.
The $\Hom$-sets in $\catC^\circ$ are given by $\Hom_{\catC^\circ}((M,e),(M',e'))=e'\Hom_{\catC}(M,M')e$ and the composition is the same
as in $\catC$. Finally, set $M^\circ=(M,\id_M)$ and $f^\circ =f$ for any object $M\in\catC$ and any morphism
$f$ in $\catC$. For simplicity, we will use only this particular realization of $\catC^\circ$. Nevertheless,
the universality implies that the statements to
follow hold for any pseudo-abelian closure.

Assume $\catC$ admits a $K$-linear hermitian structure $(*,\omega)$.
Then $\catC^\circ$ is clearly a $K$-category, and moreover, it
has a $K$-linear hermitian structure
given by $(M,e)^*=(M^*,e^*)$ and $\omega_{(M,e)}=e^{**}\omega_M e\in\Hom_{\catC^\circ}((M,e),(M^{**},e^{**}))$.
Furthermore, the functor $M\mapsto M^\circ$ is $1$-hermitian and duality preserving (the isomorphism
$(M^*)^\circ\to (M^\circ)^*$ being $\id_M$), so
we have a faithful and full functor $(M,s)\mapsto (M,s)^\circ=(M^\circ,s)$ from $\Sesq{\catC}$ to $\Sesq{\catC^\circ}$.
Henceforth, consider $\catC$ (resp.\ $\Sesq{\catC}$) as a full subcategory of $\catC^\circ$ (resp.\ $\Sesq{\catC^\circ}$), i.e\
identify $M^\circ$ (resp.\ $(M,s)^\circ$) with $M$ (resp.\ $(M,s)$).

\begin{lem}\label{LM:Karoubi}
    Let $\catC$, $\catC'$ be two hermitian categories and let $F:\catC\to \catC'$
    be an $\epsilon$-hermitian duality preserving functor. Then:
    \begin{enumerate}
        \item[(i)] $F$ extends to an $\epsilon$-hermitian duality preserving functor $F^\circ:\catC^\circ\to
        \catC'^\circ$. If $F$ is faithful and full, then so is $F^\circ$.
        \item[(ii)] There is a $1$-hermitian  duality preserving functor $G:\TDA{\catC}^\circ\to\TDA{\catC^\circ}$.
        The functor $G$ fixes $\TDA{\catC}$ and induces an equivalence of categories.
    \end{enumerate}
\end{lem}

\begin{proof}
    (i) Define $F^\circ(M,e)=(FM,Fe)\in\catC'^\circ$. The rest is routine.

    (ii)
    Let $G$ send $((M,M',f,g),(e,e'^\op))\in\TDA{\catC}^\circ$
    to $((M,e),(M,e'),e'^*fe,e'^*ge)$ and any morphism to itself.
    The details are left to the reader.
\end{proof}

Observe that the category $\catC_L$ may not be pseudo-abelian even when $\catC$ is. We thus
set $\catC_L^\circ:=(\catC_L)^\circ$.

\begin{thm}
    Let $(\catC,*,\omega)$ be a pseudo-abelian $K$-linear hermitian category such that
    $\dim_K\Hom(M,M')$ is finite
    for all $M,M'\in\catC$. Then the maps
    \[\WittF[\epsilon]{\scalarExt{L}{K}}:\Witt[\epsilon]{\catC}\to\Witt[\epsilon]{\catC_L^\circ}
    \qquad
    \textrm{and}
    \qquad
    \Witt{\scalarExt{L}{K}}:\WittS{\catC}\to\WittS{\catC_L^\circ}
    \]
    are injective.
\end{thm}

\begin{proof}
    We begin by showing that $\WittF[\epsilon]{\scalarExt{L}{K}}:\Witt[\epsilon]{\catC}\to\Witt[\epsilon]{\catC_L^\circ}$ is injective.
    Let $(M,s)\in\Herm[\epsilon]{\catC}$ be such that $(M_L,s_L)\equiv 0$ in $\Witt[\epsilon]{\catC_L^\circ}$.
    Then there are objects $N,N'\in\catC_L^\circ$ such that $s_L\oplus \Hyp[\epsilon]{N}\simeq \Hyp[\epsilon]{N'}$.
    Let $(U,h)=(M,s)\oplus (N', \Hyp[\epsilon]{N'})$, $E=\End_{\catC^\circ}(U)$ and let
    $\sigma$ be the involution induced by $h$ on $E$. Also set $E_L=E\otimes_KL=\End_{\catC^\circ_L}(U_L)$ and $\sigma_L=\sigma\otimes_K\id_L$.
    Subsection \ref{subsection:scalar-extension-commutes-with-transfer} implies
    that $\scalarExt{L}{K}(\SesqF{\Transfer{(U,h)}}(M,s))=\SesqF{\Transfer{(U_L,h_L)}}(M_L,s_L)\equiv 0$
    in $\Witt[\epsilon]{E_L,\sigma_L}$, and
    by \cite[Prp.\ 1.2]{BayLen90}, this means $\SesqF{\Transfer{(U,h)}}(M,s)\equiv 0$ in $\Witt[\epsilon]{E,\sigma}$
    (here we need $\dim_KE<\infty$).
    Since $\catC$ is pseudo-abelian, the map $\Transfer{(U,h)}:\catC|_U\to \rproj{E}$ is an equivalence of categories,
    hence the induced map $\Witt[\epsilon]{\Transfer{(U,h)}}:\Witt[\epsilon]{\catC|_U}\to \Witt{\rproj{E}}=\Witt[\epsilon]{E,\sigma}$
    is an isomorphism of groups. Therefore, $(M,s)\equiv 0$ in $\Witt[\epsilon]{\catC|_U}$. In particular, the same identity holds in
    $\Witt[\epsilon]{\catC}$.

    Now let $(M,s)\in\Sesq{\catC}$ be such
    that $(M_L,s_L)\equiv 0$ in  $\WittS{\catC_L^\circ}$.
    Then by Proposition~\ref{PR:characterization-of-differences-of-hyp-forms-II}, $(M_L,s_L)$ is hyperbolic
    in $\catC_L^\circ$ (but, a-priori, not in $\catC_L$).
    Let $F$ be the functor defined in Theorem~\ref{TH:equivalence-of-categories} and let $J$ be the functor
    $\TDA{\catC}_L\to \TDA{\catC_L}$ of Proposition~\ref{PR:F-commutes-with-scalar-extension}.
    By the lemma, there is a fully
    faithful $1$-hermitian duality preserving functor $J':=GJ^\circ:\TDA{\catC}_L^\circ\to \TDA{\catC_L^\circ}$.
    Since $(M_L,s_L)$ is hyperbolic
    in $\catC_L^\circ$, there is $Q\in\TDA{\catC_L^\circ}$ such that $F(M_L,s_L)\simeq(Q\oplus Q^*,\Hyp{Q})$.
    Let $Z(M,s):=(M,M,s^*\omega_M,s)$ and $Z(M_L,s_L)=(M_L,M_L,s_L^*\omega_{M_L},s_L)$.
    Recall that
    $F(M_L,s_L)=F\SesqF{\scalarExt{L}{K}}(M,s)=\SesqF{J}\SesqF{\scalarExt{L}{K}}F(M,s)$
    (Proposition~\ref{PR:F-commutes-with-scalar-extension}) and hence $Q\oplus Q^*\simeq Z(M_L,s_L)=J(Z(M,s)_L)=J'(Z(M,s)_L)$.
    As $J'$ is fully faithful and its image is pseudo-abelian,
    we may assume $Q=J' H$ for some $H\in \TDA{\catC}_L^\circ$.
    We now have $\SesqF{J'}(H\oplus H^*,\Hyp{H})=(Q\oplus Q^*,\Hyp{Q})\simeq F(M_L,s_L)=J' \scalarExt{L}{K}F(M,s)$, hence
    $(H\oplus H^*,\Hyp{H})\simeq \scalarExt{L}{K}F(M,s)$ in $\TDA{\catC}_L^\circ$.
    In particular, $\scalarExt{L}{K}F(M,s)\equiv 0$ in $\Witt{\TDA{\catC}_L^\circ}$. By the previous paragraph, this means
    $F(M,s)\equiv0$  in $\Witt{\TDA{\catC}}$ and hence, $(M,s)\equiv 0$ in $\WittS{\catC}$.
\end{proof}

We also have the following weaker version of Springer's Theorem that works without assuming
$\catC$ is pseudo-abelian.

\begin{thm}
    Let $(\catC,*,\omega)$ be a  $K$-linear hermitian category such that\linebreak
    $\dim_K\Hom(M,M')$ is finite
    for all $M,M'\in\catC$. Then the map
    $\WittF[\epsilon]{\scalarExt{L}{K}}:\Witt[\epsilon]{\catC}\to\Witt[\epsilon]{\catC_L}$
    is injective.
\end{thm}

\begin{proof}
    Let $(M,s)\in\Herm[\epsilon]{\catC}$ be such that $(M_L,s_L)\equiv 0$ in $\Witt[\epsilon]{\catC_L}$.
    Then there are objects $N_L,N'_L$ such that $s_L\oplus \Hyp[\epsilon]{N_L}\simeq \Hyp[\epsilon]{N'_L}$.
    Since $\Hyp[\epsilon]{N_L}=(\Hyp[\epsilon]{N})_L$ and $\Hyp[\epsilon]{N'_L}=(\Hyp[\epsilon]{N'})_L$,
    we have $(s\oplus \Hyp[\epsilon]{N})_L\simeq (\Hyp[\epsilon]{N'})_L$. By Theorem~\ref{TH:our-Springer},
    this means $s\oplus \Hyp[\epsilon]{N}\simeq \Hyp[\epsilon]{N'}$, hence $(M,s)\equiv 0$ in $\Witt[\epsilon]{\catC}$.
\end{proof}

\subsection{Weak Hasse Principle}

In this final subsection, we prove a version of the \emph{weak
Hasse principle} for systems of sesquilinear forms over hermitian categories.
Recall that the weak Hasse principle asserts that two quadratic forms
over a global field $k$ are isometric if and only if they are isometric over all completions
of $k$. This actually   fails for  systems of quadratic forms, and we refer
the reader to \cite{Bayer85} and  \cite{Bayer87} for necessary and sufficient conditions for the weak Hasse principle to hold
in this case.
A weak Hasse principle for \emph{sesquilinear} forms defined over a skew field with a unitary involution was obtained in \cite{BayerMoldovan12}.

\medskip

Let $K$ be a commutative ring admitting an involution $\sigma$, and let $k$ be the fixed ring of $\sigma$.
Let
$\catC$ be an additive $K$-category.
A hermitian structure
$(*,\omega)$ on $\catC$ is called \emph{$(K,\sigma)$-linear} if $(fa)^*=f^*\sigma(a)$ for all $a\in K$ and any morphism $f$ in $\catC$. (This
means that the functor $*$ is $k$-linear.)
In this case, $\End(M)$ is a $K$-algebra
for all $M\in \catC$, and for any unimodular $\epsilon$-hermitian form $(M,s)$ over $\catC$, the restriction
of the involution $f\mapsto s^{-1}f^*s$ to $K\cdot \id_M$ is $\sigma$.

Suppose now that $K$ is a global field of characteristic not $2$
admitting an involution $\sigma$ of the second kind with fixed field $k$, and that $\catC$ admits a nonempty family
of $(K,\sigma)$-linear hermitian structures $\{*_i,\omega_i\}_{i\in I}$.
For every prime spot $p$ of $k$, let $k_p$ be the completion of $k$ at $p$ and set $K_p=K\otimes_k k_p$,
$\sigma_p=\sigma\otimes_k \id_{k_p}$ and $\catC_p=\catC\otimes_k k_p$.
Then each of the hermitian structures $(*_i,\omega_i)$ gives rise to a $(K_p,\sigma_p)$-linear hermitian
structure on $\catC_p$, which we also denote by $(*_i,\omega_i)$.

\begin{thm}\label{TH:Hasse-principal-for-systems}
    Let $K$ be a global field of characteristic not $2$ admitting an  involution $\sigma$ \emph{of the second kind} with fixed field $k$.
    Let $\catC$ be a $K$-category such that $\dim_K\Hom(M,N)$ is finite for all $M,N\in\catC$, and assume
    there is a nonempty family $\{*_i,\omega_i\}_{i\in I}$
    of $(K,\sigma)$-linear hermitian structures on $\catC$.
    Then the weak Hasse principle (with respect to  $k$) holds  for systems of sesquilinear forms over $(\catC,\{*_i,\omega_i\})$.
    That is, two systems of sesquilinear forms over $(\catC,\{*_i,\omega_i\})$ are isometric
    if and only if they are isometric after applying $\SesqF{\scalarExt{k_p}{k}}$ for all $p$.
\end{thm}

We  will need the following lemma. (The lemma seems to be known, but we could not find an explicit reference, and
hence included here an ad-hoc proof.)

\begin{lem}\label{LM:scalar-ext-in-additive-cat}
    Let $L/K$ be any field extension, and let
    $\catC$ be an additive $K$-category such that $\dim_K\Hom_{\catC}(M,N)$ is finite
    for all $M,N\in\catC$. Then for all $N,M\in\catC$, we have $N\cong M$ $\iff$ $N_L\cong M_L$.
\end{lem}

\begin{proof}[Proof (sketch)]
    By applying $\Hom_{\catC}(M\oplus N,\,\underline{~{}~}\,)$, we may assume $M$ and $N$
    are f.g.\ projective right modules over $R:=\End(M\oplus N)$, which is a finite dimensional $K$-algebra
    by assumption.
    Let $J$ be the Jacobson radical of $R$. By tensoring with $R/J$, we may assume $R$ is semisimple.
    Let $\{V_i\}_i$ be a complete list of the simple right $R$-modules and write $(V_i)_L=\bigoplus_j W_{ij}^{n_{ij}}$
    with $\{W_{ij}\}_j$ being  pairwise non-isomorphic indecomposable $R_L$-modules.
    The $R_L$-modules $\{W_{ij}\}_{i,j}$ are pairwise non-isomorphic because
    $W_{ij}$ and $W_{i'j'}$ are non-isomorphic as $R$-modules when $i\neq i'$
    ($W_{ij}$ is isomorphic as an $R$-module to a direct sum of copies of $V_i$).
    Assume $M_L\cong N_L$ and write $M\cong\bigoplus_i V_i^{m_i}$,
    $M\cong\bigoplus_i V_i^{m'_i}$. Then
    $\bigoplus_{i,j} W_{ij}^{m_in_{ij}}\cong M_L\cong M_L\cong \bigoplus_{i,j} W_{ij}^{m'_in_{ij}}$.
    By the Krull-Schmidt Theorem (see for instance \cite[pp.\ 237 ff.]{Ro88}),
    we have $m_in_{ij}=m'_in_{ij}$ for all $i,j$, hence $m_i=m'_i$
    and $M\cong N$.
\end{proof}

\begin{proof}[Proof of Theorem~\ref{TH:Hasse-principal-for-systems}]
    By Corollary~\ref{CR:F-for-systems-preserves-isometry-after-scalar-extension}, it is enough
    to verify the Hasse principle (with respect to $k$) for $1$-hermitian forms in the category
    $\catG:=\TDA[I]{\catC}$.
    Our assumptions imply that $\catG$ is a $(K,\sigma)$-linear category
    such that $\dim_K\Hom(Z,Z')$ is finite for all $Z,Z'\in\catG$.
    We now use  the ideas developed in \cite[\S9]{BayerMoldovan12}.

    Let $(Z,h)$, $(Z',h')$
    be two unimodular $1$-hermitian forms over $\catG$ such
    that $\SesqF{\scalarExt{k_p}{k}}(Z,h)\simeq\SesqF{\scalarExt{k_p}{k}}(Z',h')$
    for all $p$. By Lemma~\ref{LM:scalar-ext-in-additive-cat}, this is implies that $Z\cong Z'$, so we may assume $Z=Z'$.

    Fix a $1$-hermitian form $h_0$ on $Z$ and let $\tau$ be the involution induced
    by $h_0$ on $E:=\End(Z)$ (i.e.\ $\tau(x)=h_0^{-1}x^*h_0$).
    There is an equivalence relation on the elements of $E$  defined by $x\sim y$ $\iff$ there exists
    an invertible $z\in E$ such that $x=zy\tau(z)$. Let $H(\tau,\units{E})$ be the set of equivalence
    classes of invertible elements $x \in \units{E}$ for which $x=\tau(x)$. In
    the same manner as in  \cite[Th.\ 5.1]{BayerMoldovan12}, we see that there is a one-to-one
    correspondence between isometry classes of unimodular \mbox{$1$-hermitian} forms on $Z$ and
    elements $H(\tau,\units{E})$. It is given by $(Z,t)\mapsto h_0^{-1}t$.

    Applying the same argument to $Z_p=\scalarExt{k_p}{k} Z\in\catG_p$, we see that the weak Hasse principle
    is equivalent to the injectivity of the standard map
    \[
    \Phi:H(\tau,\units{E})\to \prod_p H(\tau_p,\units{E_p})
    \]
    where $E_p=\End(Z_p)=E\otimes_k k_p$ and $\tau_p=\tau\otimes_k \id_{k_p}$.
    Observe that since $\catG$ is $(K,\sigma)$-linear, $\tau$ is a unitary involution (and in fact, $\tau|_K=\sigma$).
    By \cite[\S9]{BayerMoldovan12}, this means that $\Phi$ is injective, hence the weak Hasse principal
    holds.
\end{proof}

\begin{cor}
    Let $K$ be a global field of characteristic not $2$ admitting an involution $\sigma$ \emph{of the second kind} with fixed field $k$.
    Let $A$ be a finite dimensional $K$-algebra admitting a nonempty family of involutions $\{\sigma_i\}_{i\in I}$
    such that $\sigma_i|_K=\sigma$. Then
    the weak Hasse principle (with respect to $k$) holds  for systems of sesquilinear forms over $(A,\{\sigma_i\})$.
\end{cor}

\section*{Acknowledgements}

The third named author would like to thank Emmanuel Lequeu for many interesting and useful discussions.

\bibliographystyle{plain}
\bibliography{Herm_Categories_Bib}

\begin{thebibliography}{10}

\bibitem{BayKeaWil89}
E.~Bayer-Fluckiger, C.~Kearton, and S.~M.~J. Wilson.
\newblock Hermitian forms in additive categories: finiteness results.
\newblock {\em J. Algebra}, 123(2):336--350, 1989.

\bibitem{BayLen90}
E.~Bayer-Fluckiger and H.~W. Lenstra, Jr.
\newblock Forms in odd degree extensions and self-dual normal bases.
\newblock {\em Amer. J. Math.}, 112(3):359--373, 1990.

\bibitem{Bayer85}
Eva Bayer-Fluckiger.
\newblock Intersections de groupes orthogonaux et principe de {H}asse faible
  pour les syst\`emes de formes quadratiques sur un corps global.
\newblock {\em C. R. Acad. Sci. Paris S\'er. I Math.}, 301(20):911--914, 1985.

\bibitem{Bayer87}
Eva Bayer-Fluckiger.
\newblock Principe de {H}asse faible pour les syst\`emes de formes
  quadratiques.
\newblock {\em J. Reine Angew. Math.}, 378:53--59, 1987.

\bibitem{BayerMoldovan12}
Eva Bayer-Fluckiger and Daniel~Arnold Moldovan.
\newblock Sesquilinear forms over rings with involution.
\newblock {\em J. Pure Appl. Algebra}, 218(3):417--423, 2014.

\bibitem{Bj71}
J.-E. Bj{\"o}rk.
\newblock Conditions which imply that subrings of semiprimary rings are
  semiprimary.
\newblock {\em J. Algebra}, 19:384--395, 1971.

\bibitem{CaDi93}
Rosa Camps and Warren Dicks.
\newblock On semilocal rings.
\newblock {\em Israel J. Math.}, 81(1-2):203--211, 1993.

\bibitem{Fain}
L.~Fainsilber.
\newblock Formes hermitiennes sur des alg\`ebres sur des anneaux locaux.
\newblock In {\em Th\'eorie des nombres, {A}nn\'ees 1992/93--1993/94}, Publ.
  Math. Fac. Sci. Besan\c con, page~10. Univ. Franche-Comt\'e, Besan\c con,
  199?

\bibitem{Fi12}
Uriya~A. First.
\newblock Semi-invariant subrings.
\newblock {\em J. Algebra}, 378:103--132, 2013.

\bibitem{Hi60}
Yukitoshi Hinohara.
\newblock Note on non-commutative semi-local rings.
\newblock {\em Nagoya Math. J.}, 17:161--166, 1960.

\bibitem{Kar78}
Max Karoubi.
\newblock {\em {$K$}-theory}, volume 226 of {\em Grundlehren der Mathematischen
  Wissenschaften [Fundamental Principles of Mathematical Sciences]}.
\newblock Springer-Verlag, Berlin, 1978.

\bibitem{Kel88}
Bernhard Keller.
\newblock A remark on quadratic spaces over noncommutative semilocal rings.
\newblock {\em Math. Z.}, 198(1):63--71, 1988.

\bibitem{Kne69}
Manfred Knebusch.
\newblock Isometrien \"uber semilokalen {R}ingen.
\newblock {\em Math. Z.}, 108:255--268, 1969.

\bibitem{Kn91}
Max-Albert Knus.
\newblock {\em Quadratic and {H}ermitian forms over rings}, volume 294 of {\em
  Grundlehren der Mathematischen Wissenschaften [Fundamental Principles of
  Mathematical Sciences]}.
\newblock Springer-Verlag, Berlin, 1991.
\newblock With a foreword by I. Bertuccioni.

\bibitem{QuSchSch79}
H.-G. Quebbemann, W.~Scharlau, and M.~Schulte.
\newblock Quadratic and {H}ermitian forms in additive and abelian categories.
\newblock {\em J. Algebra}, 59(2):264--289, 1979.

\bibitem{Reiter75}
H.~Reiter.
\newblock Witt's theorem for noncommutative semilocal rings.
\newblock {\em J. Algebra}, 35:483--499, 1975.

\bibitem{RiSh76}
C.~Riehm and M.~A. Shrader-Frechette.
\newblock The equivalence of sesquilinear forms.
\newblock {\em J. Algebra}, 42(2):495--530, 1976.

\bibitem{Ri74}
Carl Riehm.
\newblock The equivalence of bilinear forms.
\newblock {\em J. Algebra}, 31:45--66, 1974.

\bibitem{Ro88}
Louis~H. Rowen.
\newblock {\em Ring theory. {V}ol. {I}}, volume 127 of {\em Pure and Applied
  Mathematics}.
\newblock Academic Press Inc., Boston, MA, 1988.

\bibitem{Sch85QuadraticAndHermitianForms}
Winfried Scharlau.
\newblock {\em Quadratic and {H}ermitian forms}, volume 270 of {\em Grundlehren
  der Mathematischen Wissenschaften [Fundamental Principles of Mathematical
  Sciences]}.
\newblock Springer-Verlag, Berlin, 1985.

\end{thebibliography}

\end{document}